\documentclass[12pt,reqno,a4paper]{amsart}
\usepackage{amsmath,amssymb,amsthm,amsfonts}
\usepackage{tikz}
\usepackage{hyperref}
\usepackage{appendix}
\usepackage{scalerel}

\usepackage{graphicx}
\usepackage{subcaption}
\makeatletter
\def\set@curr@file#1{%
  \begingroup
    \escapechar\m@ne
    \xdef\@curr@file{\expandafter\string\csname #1\endcsname}%
  \endgroup
}
\def\quote@name#1{"\quote@@name#1\@gobble""}
\def\quote@@name#1"{#1\quote@@name}
\def\unquote@name#1{\quote@@name#1\@gobble"}
\makeatother
\usepackage{graphics}

\DeclareRobustCommand{\mychar}[1]{%
  \begingroup\normalfont
  \includegraphics[height=\fontcharht\font`\B]{#1}%
  \endgroup
} 

\theoremstyle{plain}
\newtheorem{thm}{Theorem}[section]
\newtheorem{lem}[thm]{Lemma}
\newtheorem{prop}[thm]{Proposition}

\theoremstyle{definition}

\theoremstyle{remark}

\newcommand{\Rr}{\mathbb{R}}
\newcommand{\Zz}{\mathbb{Z}}
\newcommand{\Nn}{\mathbb{N}}

\newcommand{\etal}{\textit{et al}. }

\DeclareMathOperator*{\argmax}{arg\,max}
\DeclareMathOperator*{\BR}{BR}

\begin{document}

\title[Periodic attractor in the RPS game]
{Periodic attractor in the discrete time best-response dynamics of the Rock-Paper-Scissors game}


\author[Gaiv\~ao]{Jos\'e Pedro Gaiv\~ao}
\address{Departamento de Matem\'atica and CEMAPRE, REM, ISEG\\
Universidade de Lisboa\\
Rua do Quelhas 6, 1200-781 Lisboa, Portugal}
\email{jpgaivao@iseg.ulisboa.pt}

\author[Peixe]{Telmo Peixe}
\address{Departamento de Matem\'atica and CEMAPRE, REM, ISEG\\
Universidade de Lisboa\\
Rua do Quelhas 6, 1200-781 Lisboa, Portugal}
\email{telmop@iseg.ulisboa.pt}

\date{\today}  

\begin{abstract}
The Rock-Paper-Scissors (RPS) game is a classic non-cooperative game widely studied in terms of its theoretical analysis as well as in its
applications, ranging from sociology and biology to economics.
Many experimental results of the RPS game indicate that this game is better modelled by the discretized best-response dynamics rather than continuous time dynamics. In this work we show that the attractor of the discrete time best-response dynamics of the RPS game is finite and periodic.
Moreover we also describe the bifurcations of the attractor and determine the exact number, period and location of the periodic strategies.
\end{abstract}
 
\maketitle

\section{Introduction}\label{sec:intro}

The widely known Rock-Paper-Scissors (RPS) game consists of two players, each one throwing one hand forward making one of three possible symbols:
\begin{itemize}
	\item[(R)] Rock, represented by the closed hand;
	\item[(P)] Paper, represented by the open hand; and
	\item[(S)] Scissors, represented by the closed hand with exactly two fingers extended.
\end{itemize}
At each turn the players compare the symbols represented by their hands and decide who wins as follows:
\begin{itemize}
	\item Paper (P) wins Rock (R);
	\item Scissors (S) wins Paper (P); and
	\item Rock (R) wins Scissors (S);
\end{itemize}
forming a dominance cycle as depicted in Figure~\ref{RPS_cycle_dominance}.
It is considered a draw if both players make the same symbol.

\begin{figure}[h]
\centering{\includegraphics[width=6cm]{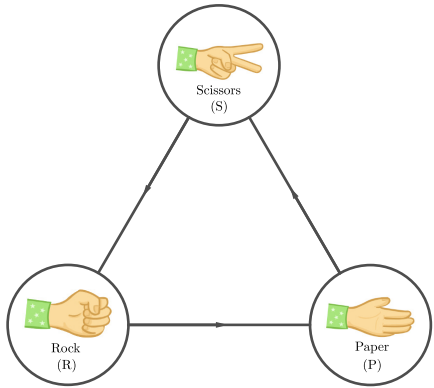}}
\caption{Cyclic dominance in the RPS game.}\label{RPS_cycle_dominance}
\end{figure}

The RPS game is thus a game with three pure strategies which in its normal form can be represented by the payoff matrix
\begin{equation}\label{RPS_payoff_matrix}
A=\begin{pmatrix}
0&-b&a\\
a&0&-b\\
-b&a&0
\end{pmatrix}, \quad a,b>0.
\end{equation}
The most commonly known version of this game is the symmetric case, where what the players win is the same as what they lose (called a zero-sum game), and which can be represented by the payoff matrix $A$ with $a=b=1$. In the case where $a>b$ we say that the game is \textit{favourable}, since in this situation what the players can win is greater than what they can loose. In the other case where $a<b$, we say that the game is \textit{unfavourable} because what the players can win is less than what they can loose.
A game with a more general payoff matrix where all $a$'s and $b$'s are different is analyzed in~\cite{GauHof1995}.

The RPS game is often used as a model for studying the evolution of competitive strategies (non-cooperative games) in dominance cycles~\cite{SMJSRP2014}.
Namely, in the theoretical economics field this game has been used to qualitatively study price dynamics~\cite{HopSey2002}.
A widely studied case of cyclical dominance in economics is the designated~\textit{Edgworth price cycle}~\cite{MasTir1988} which describes the cyclical pattern of price changes in a given market, such as the retail price cycles in the gasoline market~\cite{Noel2007}.

In evolutionary game theory we can study the evolution of a given game using different models,
such as the replicator equation or the best-response function.
Suppose that in a given population some individuals have the capability to change their strategy at any time,
switching to a strategy that is the best-response to their opponents' current strategy.
A function that models this situation is the classical~\textit{Brown-Robinson procedure},
or~\textit{fictitious play}, initially studied by G. Brown~\cite{Brown1951} and J. Robinson~\cite{Rob1951}.
Brown studied two versions of the~\textit{fictitious play} model, the discrete and the continuous time.
The continuous time version, up to a time re-parametrisation that only affects the velocity of the motion,  is given by the differential inclusion
$$
x'(t) \in \BR(x(t))-x(t),
$$
where $ BR(x(t))$ is the best-response to the strategy $x(t)$ at time $t$.
This is the designated~\textit{Best-Response dynamics} (BR)~\cite{Matsui1992}.

For the BR dynamics of the RPS game, by~\cite[Theorem 2]{GauHof1995} we have that when $a\ge b$ (favourable and zero-sum games), the Nash equilibrium is globally attracting, and when $a <b$ (unfavourable game), the Nash equilibrium is repelling and the global attractor is a periodic orbit known as the~\textit{Shapley triangle}.

Suppose now that in a short period of time $\varepsilon$ a small fraction $\varepsilon$ of randomly chosen people from the population
can change their strategy for a better strategy $\BR(x(t))$ relative to the current state $x(t)$ of the population.
In this case we have
\begin{equation}\label{disc_BR}
x(t+ \varepsilon ) \in (1- \varepsilon )x(t)+ \varepsilon \BR(x(t)),
\end{equation}
which is a discretization of the BR dynamics.
This discretized version has been studied by several other authors, e.g., J. Hofbauer and its collaborators~\cite{HS2006, BHS2012, BH2017},
D. Monderer~\etal\cite{MSS1997}, C. Harris~\cite{Har1998}, and V. Krishna and T. Sj\"{o}str\"{o}m~\cite{KS1998}.

Many experiments (e.g.,~\cite{SKM2003, CFH2014, WXZ2014}) on the evolution of the choices of strategies that each person makes as they play the RPS game evidence a dynamics that is well modelled by the discrete time best-response rather than the continuous time version.

The work we present here was also motivated by the paper of P. Bednarik and J. Hofbauer~\cite{BH2017} where they study the discretized best-response dynamics for the RPS game.
They focus on the symmetric case with $a=b=1$, making in the end some extension to the general case $a,b>0$.
In their main result, they prove that the attractor is contained in an annulus shaped triangular region and find a family of periodic strategies inside that region. 

In this paper we also study the discretized best-response dynamics for the RPS game.
We consider the general case $a,b>0$ and our main result says that the attractor of the discretized best-response dynamics for the RPS game is made of a finite number of periodic strategies, as stated in the following theorem.

\begin{thm}\label{th:main}
The attractor of the discretized best-response dynamics of the rock-paper-scissors game is finite and periodic, i.e., every strategy converges to a periodic strategy and there are at most a finite number of them.
\end{thm}

We also describe in detail the bifurcations of the attractor and determine the exact number, period and location of the periodic strategies. See Theorem~\ref{thm:bif} for a full description. For instance, we show that every periodic strategy has a period which is a multiple of 3 and the attractor is formed by a chain of such periodic strategies enumerated by their period, i.e., with consecutive periods. Denote by $N(\varepsilon)$ the number of distinct periodic strategies (which also depends on $a$ and $b$). In the zero-sum game ($a=b$), we show that $N(\varepsilon)$ grows without bound as $\varepsilon\to0$. In fact,
$$
N(\varepsilon)=\left\lceil\frac{\log\left(\frac{2+\varepsilon-\sqrt{3\varepsilon(4-\varepsilon)}}{2(1-\varepsilon)}\right)}{\log(1-\varepsilon)}\right\rceil-1\quad\text{and}\quad \lim_{\varepsilon\to0}N(\varepsilon)=\infty.
$$
Moreover, the attractor is formed by the union of periodic strategies having periods $3,6,\ldots,3N(\varepsilon)$,  converging to the Nash equilibrium of the game.

In the non-zero-sum game ($a\neq b$) we show that the number $N(\varepsilon)$ of periodic strategies stays bounded as $\varepsilon\to0$. Although more complicated, similar formulas for $N(\varepsilon)$ are known in the non-zero-sum case. For instance, when $a>b$ (favourable game), we know that
$$\lim_{\varepsilon\to0}N(\varepsilon)=\left\lceil\frac{3a}{a-b}\right\rceil-1.$$ 
This fact is rather surprising as the authors of \cite[pag. 84]{BH2017} write: "Overall, if $a>b$, the dynamics behaves qualitatively very similarly to the case
$a = b$ : more and more periodic orbits emerge, as $\varepsilon\to0$". 
When $a<b$  (unfavourable game), the number $N(\varepsilon)$ has no limit as $\varepsilon\to0$, i.e., it alternates between the integers $\left\lceil\frac{b}{b-a}\right\rceil-1$ and $\left\lceil\frac{b}{b-a}\right\rceil$ for every $\varepsilon$ sufficiently small. Unlike before, in this case the periods of the periodic strategies grow like $3\log(b/a)/\varepsilon$ as $\varepsilon\to 0$\footnote{In fact, the smallest period is $3\left(\left\lfloor\frac{\log(a/b)}{\log(1-\varepsilon)}\right\rfloor +1\right)\sim \frac{3\log(b/a)}{\varepsilon}$ as $\varepsilon\to0$.}. 
 Moreover, when the game is favourable ($a>b$), the attractor converges to the Nash equilibrium, whereas in the unfavourable case ($a<b$), the attractor converges to the Shapley triangle. 

The rest of the paper is organized as follows. In Section~\ref{sec:prelim} we introduce the best-response function and define the map $T$ whose dynamics we a\-na\-ly\-se in this paper. We also introduce the general Rock-Paper-Scissors game.
In Section~\ref{sec:ReducSym}, using the symmetry of the game we reduce the dynamics to a subregion of the phase space.
In Section~\ref{sec:PoincMap} we define a Poincar\'e map and prove some of its properties in Appendix~\ref{sec:MonotLems}.
In Section~\ref{sec:MainResProof} we prove Theorem~\ref{th:main} using results from the previous Sections~\ref{sec:ReducSym} and~\ref{sec:PoincMap}.
In Section~\ref{sec:PerOrbBif} we describe in detail the bifurcations of the attractor and in Section~\ref{sec:Conclus} we discuss related work and some open problems.


\section{Preliminaries}\label{sec:prelim}

\subsection{Best-Response dynamics}
A population game with $d\geq2$ pure strategies is determined by a payoff matrix $A=(a_{i,j})\in \Rr^{d\times d}$ where $a_{i,j}$ is the payoff of the population playing the pure strategy $i$ against the pure strategy $j$. A mixed strategy is a probability vector $x\in\Delta$ where $\Delta$ denote the $(d-1)$-dimensional simplex,
$$
\Delta=\left\{(x_1,\ldots,x_d)\in\Rr^d\colon x_1+\cdots+x_d=1,\quad x_i\geq0\right\}.
$$
The best-response against $y\in\Delta$ is
$$
\BR(y)=\argmax_{x\in\Delta}x^\top Ay.
$$
The best-response is a multivalued map defined on $\Delta$. Indeed, on the indifferent sets (see Figure~\ref{RPS_regionsR_linesGamma})
$$
\Gamma_{i,j}=\{y\in\Delta\colon (Ay)_i=(Ay)_j\geq (Ay)_k,\quad \forall\,k=1,\ldots,d,\quad i\neq j \}\,,
$$
the $\BR$ takes any value in the convex combination of the canonical vectors $e_i$ and $e_j$. On the complement set $\Delta\setminus \Gamma$ where $\Gamma=\bigcup_{i\neq j}\Gamma_{i,j}$, the best-response $\BR$ is piecewise constant, i.e., takes the single value $e_i$ in the interior of $\BR^{-1}(\{e_i\})$ relative to $\Delta$. 

Let $T\colon\Delta\setminus\Gamma\to\Delta$ denote the map
$$
T(x)=\lambda x + (1-\lambda)\BR(x),
$$
where $\lambda\in (0,1)$. This map is deduced from~\eqref{disc_BR} considering $\lambda=1-\varepsilon$.


In this paper we are interested in the dynamics of the map $T$. To that end we need to introduce some terminology borrowed from dynamical systems with discontinuities. A point $x\in\Delta$ is called \textit{regular} if $T^n(x)\notin\Gamma$ for every $n\geq0$. The set of regular points is denoted by $\hat{\Delta}$. For a generic payoff matrix $A$, the set of regular strategies $\hat{\Delta}$ equals the set $\Delta$ except for a countable number of co-dimension one hyperplanes. Thus, $\hat{\Delta}$ is a full measure (with respect to Lebesgue measure on $\Delta$) and residual subset of $\Delta$. From a dynamical systems point of view, we aim to describe the orbits of almost every point in $\Delta$. This is the main reason for considering $T$ as a piecewise single-valued map and not a multi-valued correspondence on the whole $\Delta$.

Given $x\in\hat{\Delta}$ we denote by $O_T(x)$ its orbit, i.e., the sequence $O_T(x)=\{T^n(x)\}_{n\geq0}$. We say that a $x\in\hat{\Delta}$ is \textit{periodic} with period $p\in\Nn$ if $T^p(x)=x$ and $T^k(x)\neq x$ for every $1\leq k<p$. The $\omega$-limit set of a regular point $x\in\hat{\Delta}$ is the set of limit points of its orbit. The \textit{attractor} of $T$, which we denote by $\Lambda$, is the closure of the union of $\omega(x)$ over all $x\in\hat{\Delta}$. We say that $T$ is \textit{asymptotically periodic} if it has at most a finite number of periodic regular points and the $\omega$-limit set of every regular point is a periodic orbit. In other words, $T$ is asymptotically periodic whenever its attractor $\Lambda$ is finite and periodic.

\subsection{Rock-Paper-Scissors game}Let $A$ be the payoff matrix of the RPS game introduced in Section~\ref{sec:intro}. Define 
$$
\alpha:=\frac{a}{b},
$$
where $a$ and $b$ are the parameters of the payoff matrix $A$. Notice that $\alpha>0$ for every $a,b>0$. The symmetric case $a=b$ corresponds to $\alpha=1$.


\begin{figure}[h]
\centering{\includegraphics[width=7cm]{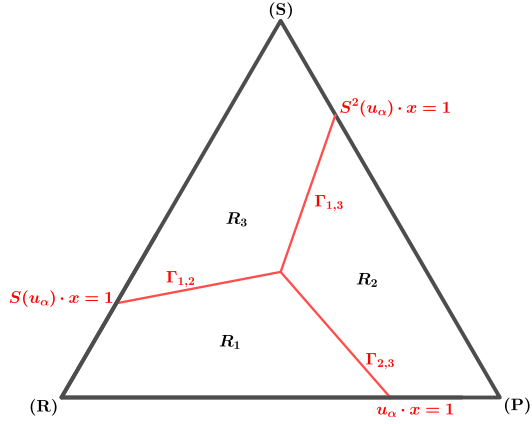}}
\caption{The simplex $\Delta$ with the regions $R_i$ and the indifference sets $\Gamma_{i,j}$.}
\label{RPS_regionsR_linesGamma}
\end{figure}

The domain $\Delta\setminus\Gamma$ of the map $T$ is the union of three disjoint regions (see Figure~\ref{RPS_regionsR_linesGamma})
\begin{align*}
R_1&= \left \{ x\in\Delta \colon x_1>\tfrac{(\alpha-1) x_2+1}{\alpha+2} \land x_3<\tfrac{(\alpha-1) x_1+1}{\alpha+2} \right \},\\
R_2&= \left \{ x\in\Delta \colon x_2>\tfrac{(\alpha-1) x_3+1}{\alpha+2} \land x_1<\tfrac{(\alpha-1) x_2+1}{\alpha+2} \right \},\\
R_3&= \left \{ x\in\Delta \colon x_3>\tfrac{(\alpha-1) x_1+1}{\alpha+2} \land x_2<\tfrac{(\alpha-1) x_3+1}{\alpha+2} \right \},\\
\end{align*}
and the map $T$ restricted to $R_i$ has the expression
$$
T(x)=\lambda x + (1-\lambda)e_{i+1},\quad x\in R_i \,,
$$
where $i+1$ is taken module 3 and $\{e_i\}$ denotes the canonical basis of $\Rr^3$.
Let $S:\Delta\to\Delta$ be the map $(x_1,x_2,x_3)\mapsto(x_2,x_3,x_1)$. Clearly, $S$ leaves $\Delta\setminus\Gamma$ invariant. Indeed, $S(R_{i+1})=R_i$. Moreover,
\begin{lem}
$$S\circ T=T\circ S$$
\end{lem} 
\begin{proof}
For any $x\in R_i$,
$$
S\left(T(x)\right)=S\left(\lambda x+(1-\lambda)e_{i+1}\right)=\lambda S(x)+(1-\lambda)e_i=T\left(S(x)\right)\,.
$$
\end{proof}
This means that $S$ is a symmetry for $T$. In the following section we will use this symmetry to reduce the study of the dynamics of $T$ to a single map with domain $R_1$.

Let 
\begin{equation}\label{ualpha}
u_\alpha:=(\alpha+2,1-\alpha,0).
\end{equation}
Then, we can express $R_1$ in a more compact way
$$
R_1=\{x\in\Delta\colon u_\alpha\cdot x >1\land S(u_\alpha)\cdot x <1\}
$$
where $\cdot$ denotes the standard inner-product in $\Rr^3$. Similar expressions hold for $R_2$ and $R_3$.

\section{Reduction by symmetry}\label{sec:ReducSym}

Using the symmetry $S$ we construct a skew-product map $F$ which has the same dynamics as $T$. We proceed as follows.

Let $ \pi \colon \bigcup_{i} R_{i} \to R_{1} $ defined by $ \pi(x) = S^{i-1}(x) $ when $ x \in R_{i} $ and define $$ \Gamma_1 = R_{1} \cap  T^{-1}(\Gamma). $$ 

Next, let $ f \colon R_{1} \setminus \Gamma_1 \to R_{1} $ be the map defined by $ f(x) = \pi \circ T(x) $ for $ x \in R_{1} \setminus \Gamma_1 $. Clearly, $ R_{1} \setminus \Gamma_1 $ is the union of two regions 

\begin{align*}
A &= \{x\in R_1\colon u_\alpha\cdot x >\alpha(\lambda^{-1}-1)+1\}, \\
B &= \{x\in R_1\colon u_\alpha\cdot x <\alpha(\lambda^{-1}-1)+1\},
\end{align*}
with common boundary equal to $ \Gamma_1 $ (see Figure~\ref{Region_R1_withAB}). Recall that $u_\alpha$ is defined in \eqref{ualpha}.


\begin{figure}[h]
\centering{\includegraphics[width=8cm]{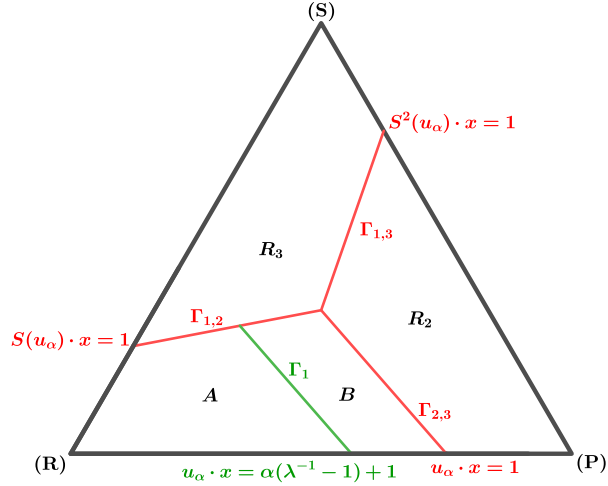}}
\caption{The simplex $\Delta$ with the regions $R_i$, where we can see the region $R_1$ divided into subregions $A$ and $B$.}
\label{Region_R1_withAB}
\end{figure}

Moreover, $ f $ is an affine transformation on both $A$ and $B$. Indeed,
$$
f|_A(x)=\lambda x + (1-\lambda) e_2  \quad \textrm{and} \quad f|_B(x)=\lambda S(x)+(1-\lambda) e_1 .
$$
Note that the sets $A$ and $B$ are characterized by the property $ T(A) \subset R_{1} $ and $T(B)\subset R_2$.
By construction,
$$ f \circ \pi = \pi \circ T. $$
Moreover, it is clear that $x$ is a regular point in $R_1$ if and only if $f^n(x)\notin \Gamma_1$ for every $n\geq0$. 

Let $\Zz_3$ denote the additive group of the integers module 3. Define $ \sigma \colon R_{1} \setminus \Gamma_1 \to \Zz_{3} $ by $ \sigma(x) = 0 $ whenever $ x \in A $ and $\sigma(x) = 1 $ otherwise.
Then, the skew-product $ F \colon (R_{1} \setminus \Gamma_1) \times \Zz_{3} \to R_{1} \times \Zz_{3} $ is given by
\[
F(x,j) = (f(x),\sigma(x) + j ) , \qquad (x,j) \in   (R_{1} \setminus \Gamma_1) \times \Zz_{3}.
\]

\begin{lem}\label{lem:conjugated}
The maps $ F $ and $ T $ are conjugated, i.e., there is a bijection $h:\Delta\setminus\Gamma\to R_1\times\Zz_3$ such that 
$$
F\circ h=h\circ T.
$$
\end{lem}

\begin{proof}
The bijection is  $h(x)=(\pi(x),i-1)$ whenever $x\in R_i$ with $i\in\{1,2,3\}$. It is now a simple computation to check that $F$ and $T$ are conjugated under $h$.
\end{proof}
By abuse of notation, we also use the symbol $ \pi $ for the projection from $ R_{1} \times \Zz_{3} \to R_1 $ given by $ \pi(x,i) = x $. 

Every periodic orbit of $ F $ (and so of $ T $) is mapped by $ \pi $ into a single periodic orbit of $ f $. The converse relation is less obvious, and is clarified in the next lemma.  For regular $x\in R_1$ let
$$ 
\sigma^{n}(x) = \sigma(f^{n-1}(x)) + \cdots + \sigma(x)\pmod{3},\quad n\geq1.
$$



\begin{lem}
\label{le:periodic}
Suppose that $ x \in R_{1} $ is a regular periodic point of $ f $ with period $p$ and that $\sigma^p(x)\neq0$. Then,  $ \pi^{-1}(O_{f}(x)) $ is a single periodic orbit of $ F $ of period $3p $.
\end{lem}
\begin{proof}
Note that, $ F^{n}(x,j)=(f^{n}(x),\sigma^{n}(x) + j ) $. Let $ y \in \pi^{-1}(O_{f}(x)) $. Then $ y = (x',i) $ for some $ x' \in O_{f}(x) $ and some $ i\in\Zz_3 $. For every integer $ j>0 $, we have 
\[ 
F^{jp}(y) = (f^{jp}(x'),j \sigma^{p}(x')+i) = (x',j \sigma^{p}(x) + i).
\]
It follows that $ y $ is a periodic point of $ F $ of period equal to 
$$
\min \{j\geq1 \colon j \sigma^{p}(x) = 0 \}\cdot p = 3p,
$$
because $\sigma^{p}(x) \in\Zz_3\setminus\{0\}$ by hypothesis.
This shows that every element of $ \pi^{-1}(O_{f}(x)) $ is a periodic point of $ F $ of period $ 3p $. But $ \pi^{-1}(O_{f}(x)) $ consists of $ 3p $ elements, so we conclude that $ \pi^{-1}(O_{f}(x)) $ is a single periodic orbit of $ F $ of period $3p$.

\end{proof}

\begin{lem}\label{lem:sufconditionAP}
If $f$ is asymptotically periodic, then $T$ is asymptotically periodic.
\end{lem}

\begin{proof}
Since $f$ is asymptotically periodic, its attractor $\Lambda_f$ consists of a finite number of periodic regular points.
As seen in the proof of Lemma~\ref{le:periodic}, every periodic regular point in $\Lambda_f$ gives at most 3 periodic regular points of $F$. Now suppose that $x$ is a regular point of $F$. Then $\pi(x)$ is a regular point of $f$. Because $f$ is asymptotically periodic, the orbit of $\pi(x)$ converges to a  periodic regular point $z\in\Lambda_f$ of $f$. 
Thus, $x$ converges to a periodic regular point in $\pi^{-1}(z)$. This shows that $F$ is asymptotically periodic. Since $F$ and $T$ are conjugated (see Lemma~\ref{lem:conjugated}), we conclude that $T$ is also asymptotically periodic.
\end{proof}

\section{Poincar\'e map}\label{sec:PoincMap}

In this section we induce the dynamics of $f$ on the region $B$. 
Define the first return time $n\colon  R_1\setminus \Gamma_1 \to \Nn$ to the closure of the set $B$ by 
$$
n(x)=\min\{k\in\Nn\colon f^k(x)\in \overline{B}\}.
$$

\begin{lem}\label{lem:return time}
The first return time is bounded, i.e., there is $C=C(\alpha,\lambda)>0$ such that $
n(x)\leq C$ for every $x\in R_1\setminus \Gamma_1 $.
\end{lem}

\begin{proof}
Let $x\in R_1\setminus \Gamma_1$. Since $f|_A(x)$ is a convex combination of $x$ and $e_2$, to determine the upper bound  $C(\alpha,\lambda)$
we have to find the first $n\in\Nn$ such that $\left(f|_A\right)^n(e_1)\in \overline{B}$, that is the first $n\in\Nn$ such that
$$
u_\alpha\cdot\left(f|_A\right)^n(e_1)\leq\alpha(\lambda^{-1}-1)+1\,,
$$
which is equivalent to $\lambda^{n+1}\leq\frac{\alpha}{2\alpha+1}.$
Hence $C:=\lceil -1+\log_{\lambda}\left(\frac{\alpha}{2\alpha+1}\right)\rceil$.
\end{proof}

Since the first return time is bounded, we can define the Poincar\'e map $P: B\to B$ by $P(x)=f^{n(x)}(x)$. Let
$$
B_k=\mathrm{int}(n^{-1}({k})\cap B).
$$
Clearly, $P|_{B_k}=(f|_A)^{k-1}\circ (f|_{B_k})$. By Lemma~\ref{lem:return time}, there are only finitely many non-empty sets $B_k$. This means that $P$ is piecewise affine with a finite number of branches. The following lemma gives an analytic expression for these sets $B_k$
(see Figure~\ref{Region_B_withBks}) and corresponding restrictions of the Poincar\'e map.
Let
$$
b_k:= \lambda^{-k-1}\alpha - \lambda^{-1}(2\alpha +1) + \alpha + 2.
$$


\begin{figure}[h]
\centering{\includegraphics[width=9cm]{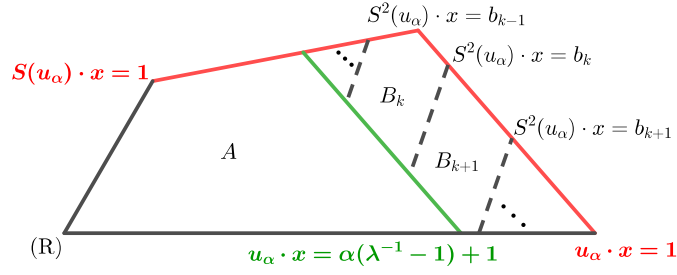}}
\caption{The region $R_1$, where $B$ is divided into subregions $B_k$ of $B$ delimited by $S^2(u_\alpha)\cdot x=b_{k-1}$ and
$S^2(u_\alpha)\cdot x=b_{k}$.}
\label{Region_B_withBks}
\end{figure}

\begin{lem}\label{lem:Bk and P}
For any $k\in\Nn$, 
\begin{align*}
B_k&=\{x\in B\colon b_{k-1}<S^2(u_\alpha)\cdot x < b_k\},\\
P|_{B_k}(x)&=\lambda^k S(x)+\lambda^{k-1}(1-\lambda)e_1+(1-\lambda^{k-1})e_2\,.
\end{align*}
\end{lem}

\begin{proof}
Let $k\in\Nn$. By the definition of $B_k$ we have that
$$
B_k=\{x\in B\colon f^k(x)\in B \land f^j(x)\in A, \forall j<k \}.
$$
Hence, for $x\in B_k$,
\begin{align*}
P(x)&=(f|_A)^{k-1}\circ (f|_B)(x)\\
		&=(f|_A)^{k-1}\left( \lambda S(x)+(1-\lambda) e_1 \right)\\
		&=\lambda^{k-1}\left( \lambda S(x)+(1-\lambda) e_1 \right) + (1-\lambda^{k-1})e_2\\
		&=\lambda^k S(x)+\lambda^{k-1}(1-\lambda)e_1+(1-\lambda^{k-1})e_2\,.
\end{align*}
Since $P(x)\in B$,
$$
u_\alpha\cdot P(x) \in\left(1,\alpha(\lambda^{-1}-1)+1\right),
$$
but
\begin{align*}
u_\alpha\cdot P(x)&= u_\alpha\cdot\left( \lambda^k S(x)+\lambda^{k-1}(1-\lambda)e_1+(1-\lambda^{k-1})e_2 \right)\\
&= \lambda^k S^2(u_\alpha)\cdot x+\lambda^{k-1}(1-\lambda)(\alpha+2)-(1-\lambda^{k-1})(\alpha-1)\\
&= \lambda^k S^2(u_\alpha)\cdot x+\lambda^{k-1}(\alpha-1+(1-\lambda)(\alpha+2))-\alpha+1,
\end{align*}
then
$$
\lambda^{-k}\alpha< S^2(u_\alpha)\cdot x +C_{\alpha,\lambda}< \lambda^{-k-1}\alpha,
$$
with $C_{\alpha,\lambda}=\lambda^{-1}(\alpha-1+(1-\lambda)(\alpha+2))$, and
$b_k=\lambda^{-k-1}\alpha-C_{\alpha,\lambda}$.
\end{proof}


Next, we will describe the monotonic dynamics of the Poincar\'e map $P$. Define
$$
m=m_\alpha(\lambda):=\min\{k\in\Nn\colon \lambda^k<\alpha\}.
$$
Notice that $m=1$ whenever $\alpha\geq1$. Clearly,
$$
m=1+\max\left\{0,\left\lfloor \frac{\log{\alpha}}{\log\lambda}\right\rfloor\right\}.
$$

Denote by $\hat{B}$ the set of regular points in $B$, i.e., $\hat{B}=B\cap \hat{\Delta}$.  
Given any $x\in \hat{B}$ we define its itinerary $\boldsymbol{i}(x)=(i_0,i_1,i_2,\ldots)$ where $i_k=n(P^k(x))$ for every $k\geq0$.
We describe the monotonicity of the Poincar\'e map $P$ in the following proposition.

\begin{prop}\label{prop:monotone}
Let $x\in \hat{B}$ with itinerary $\boldsymbol{i}(x)=(i_0,i_1,i_2,\ldots)$. For every $k\geq0$ the following holds:
\begin{enumerate}
\item If $i_k\geq m$, then $i_{k+j}\geq m$ for every $j\geq 0$.
\item If $i_{k}\geq m$ and $i_{k+1}\leq i_k$, then $i_{k+2}\leq i_{k+1}$.
\item  If $i_{k+1}< m$ and $i_{k+1}\geq i_k$, then $i_{k+2}\geq i_{k+1}$.
\end{enumerate}
\end{prop}

\begin{proof}
Immediate from Lemmas~\ref{lem:mon1}, \ref{lem:mon2} and \ref{lem:mon3} stated and proved in Appendix~\ref{sec:MonotLems}. 
\end{proof}

\section{Proof of Theorem~\ref{th:main}}\label{sec:MainResProof}
In this section we show that $f$ is asymptotically periodic. This implies that $T$ is asymptotically periodic by Lemma~\ref{lem:sufconditionAP}, thus proving the main Theorem \ref{th:main}. 

Let $x\in \hat{R}_1:=R_1\cap\hat{\Delta}$ be a regular point. By Lemma~\ref{lem:return time}, there is $n_0\geq0$ such that $y:=f^{n_0}(x)\in \hat{B}$. Let $\boldsymbol{i}(y)=(i_0,i_1,i_2,\ldots)$ be the itinerary of $y$ under the Poincar\'e map $P$. Notice, by Lemma~\ref{lem:return time}, that $i_k\in\{1,\ldots, C\}$ for every $k\geq0$ where $C=C(\alpha,\lambda)$ is the constant in that lemma. We will show that the itinerary of $y$ \textit{stabilizes}, i.e., there is $q\geq0$ such that $i_{k+q}=i_q$ for every $k\geq0$. We have two cases:


\begin{enumerate}
\item First suppose that $i_k<m$ for every $k\geq0$. Then either the itinerary is non-increasing, which implies that the it stabilizes or else there is $p\geq0$ such that $i_{p+1}\geq i_p$. In the former case, by item (3) of Proposition~\ref{prop:monotone}, we conclude that the itinerary is eventually non-decreasing, thus also stabilizes.
\item Now suppose that there is some $q\geq0$ such that $i_q\geq m$. Then, by item (1) of Proposition~\ref{prop:monotone}, we have that $i_{q+k}\geq m$ for every $k\geq 0$. So, arguing as in the first case, either the itinerary is non-decreasing, or else, by item (2) of Proposition~\ref{prop:monotone}, the itinerary is eventually non-increasing. In either way we conclude that the itinerary stabilizes.
\end{enumerate}
Because the itinerary of $y$ stabilizes, it means that after some iterate the orbit of $y$ belongs to $B_p$ for some $p\geq0$, i.e., there is $n_1>0$ such that $P^{n}(y)\in B_p$ for every $n\geq n_1$. Since $P$ restricted to $B_p$ is an affine contraction (see Lemma~\ref{lem:Bk and P}), we deduce that the orbit of $y$ under the map $P$ converges to the unique fixed point of $P|_{B_p}$, which we denote by $w_p$. Notice that $w_p$ belongs to the closure of $B_p$. We claim that $w_p\in B_p$. Indeed, if that was not the case, i.e., $w_p\in\partial B_p$, then as the orbit of $y$ accumulates at $w_p$ and the map $P|_{B_p}$ rotates points by an angle $2\pi/3$ about $w_p$, the orbit of $y$ would have infinitely many points outside $B_p$, which contradicts the fact that the orbit of $y$ stabilizes. Going back to the map $f$, the fixed point $w_p$ corresponds to a periodic orbit $\gamma_p:=\{w_p,f(w_p),\ldots, f^{p-1}(w_p)\}$ of $f$ having period $p$ and the orbit of $x$ converges to $\gamma_p$. This shows that every regular point $x\in \hat{R}_1$ converges under the map $f$ to a periodic orbit. Since any periodic point of $f$ corresponds to a unique fixed point of $P$, we have shown that the attractor of $f$ consists of a finite number of periodic orbits, which are fixed points of the Poncar\'e map $P$. Thus, $f$ is asymptotically periodic.

\section{Bifurcation of periodic orbits}\label{sec:PerOrbBif}

In this section we study the bifurcations of the attractor of $T$. 
Given $k\in\Nn$, let $P_k$ be the extension of $P|_{B_k}$ (see Lemma~\ref{lem:Bk and P}) to an affine contraction on $\Rr^3$ and denote by $w_k$ its unique fixed point.
\begin{lem}\label{lem: w_k}
$$
w_k=\left(\frac{\lambda^{k-1}(1-\lambda^k)}{1-\lambda^{3k}},\frac{1-\lambda^{k-1}+\lambda^{3k-1}-\lambda^{3k}}{1-\lambda^{3k}},\frac{\lambda^{2k-1}(1-\lambda^k)}{1-\lambda^{3k}}\right).
$$
\end{lem}
\begin{proof}
It is straightforward to check that $w_k=P_k(w_k)$.
\end{proof}
The points $w_k$ may not be fixed points for the Poincar\'e map $P$ as one has to check that $w_k\in B_k$ and $B_k\neq\emptyset$. However, as the following lemma shows, $P$ has no other periodic points besides the fixed points of the branch maps of $P$. 

\begin{lem}\label{}
If $x\in B$ is a periodic point of $P$, then it has period one, i.e., it is a fixed point of $P$. Moreover, $x=w_k$ for some $k\in\Nn$.
\end{lem}

\begin{proof}
In the proof of Theorem~\ref{th:main} it is shown that the orbit of any regular point  converges, under the Poincar\'e map $P$, to a fixed point of $P$. Hence, any periodic point of $P$ has to be a fixed point of a branch map $P|_{B_k}$ for some $k\in\Nn$. 
\end{proof}

In the parameter plane $(\alpha,\lambda)\in \mathcal{P}:=\Rr^+\times(0,1)$ consider the regions $\mathcal{R}_k$, with $k\in\Nn$, defined by
$$
\mathcal{R}_k:=\left\{(\alpha,\lambda)\in\mathcal{P}\colon \lambda^k<\alpha\quad \text{and}\quad q_{\alpha,\lambda}(\lambda^k) <0\right\},
$$
where $q_{\alpha,\lambda}$ is the quadratic polynomial  
$$
q_{\alpha,\lambda}(x)=(\alpha-\lambda^{-1}(\alpha-1)) x^2+(\alpha- \lambda ^{-1}(2 \alpha +1))x+\alpha.
$$
\begin{lem}\label{lem: w_k fixed}
$w_k$ is a fixed point of $P$ if and only if $(\alpha,\lambda)\in \mathcal{R}_k$.
\end{lem}

\begin{proof}
By Lemma~\ref{lem:Bk and P}, we have that $w_k\in B_k$ if and only if
$$b_{k-1}<S^2(u_\alpha)\cdot w_k < b_k.$$
Using the expression in Lemma~\ref{lem: w_k} for $w_k$ it is straightforward to see that these two inequalities define the set $\mathcal{R}_k$. Since $B_k$ might be empty, it remains to show that in fact $w_k\in B$ whenever $(\alpha,\lambda)\in \mathcal{R}_k$. By the definition of $B$  we have that $w_k\in B$ if and only if
$$1<u_\alpha\cdot w_k < \lambda^{-1}\alpha -\alpha +1 \quad\text{and}\quad S(u_\alpha)\cdot w_k<1.$$
By the inequalities that define $\mathcal{R}_k$ it is straightforward to see that
$$q_{\alpha,\lambda}(\lambda^k) <0 \Leftrightarrow u_\alpha\cdot w_k >1 $$
and
$$ \lambda^k-\alpha<0 \Leftrightarrow u_\alpha\cdot w_k < \alpha\left(\lambda^{-1} -1\right) +1 .$$
The other inequality is proved in a similar way.
\end{proof}

The region $\mathcal{R}_1$ is defined by $\lambda<\alpha$ since the inequality $q_{\alpha,\lambda}(\lambda)<0$ holds true for every $(\alpha,\lambda)\in\mathcal{P}$. Regarding the second region $\mathcal{R}_2$ we have,
$$
\mathcal{R}_2=\left\{(\alpha,\lambda)\in \mathcal{P}\colon \lambda^2<\alpha \wedge \left(\lambda\geq \lambda_* \vee \alpha< \frac{\lambda(1+\lambda)}{1-\lambda-\lambda^3}\right)\right\},
$$
where $\lambda_*=0.682328...$ is the unique real root\footnote{The inverse of the supergolden ratio.} of the polynomial $1-\lambda-\lambda^3$.
For the remaining regions, $k\geq3$, we have
$$
\mathcal{R}_k=\left\{(\alpha,\lambda)\in \mathcal{P}\colon\lambda^k<\alpha<\frac{\lambda^{k-1}(1-\lambda^k)}{1-\lambda^{k-1}-(1-\lambda)(\lambda^{2k-1}+\lambda^{k-1})}\right\}.
$$
Using this description of the regions $\mathcal{R}_k$ we obtain Figure~\ref{fig:regions_R_k}.

\begin{figure}[h]
\centering{\includegraphics[width=10cm]{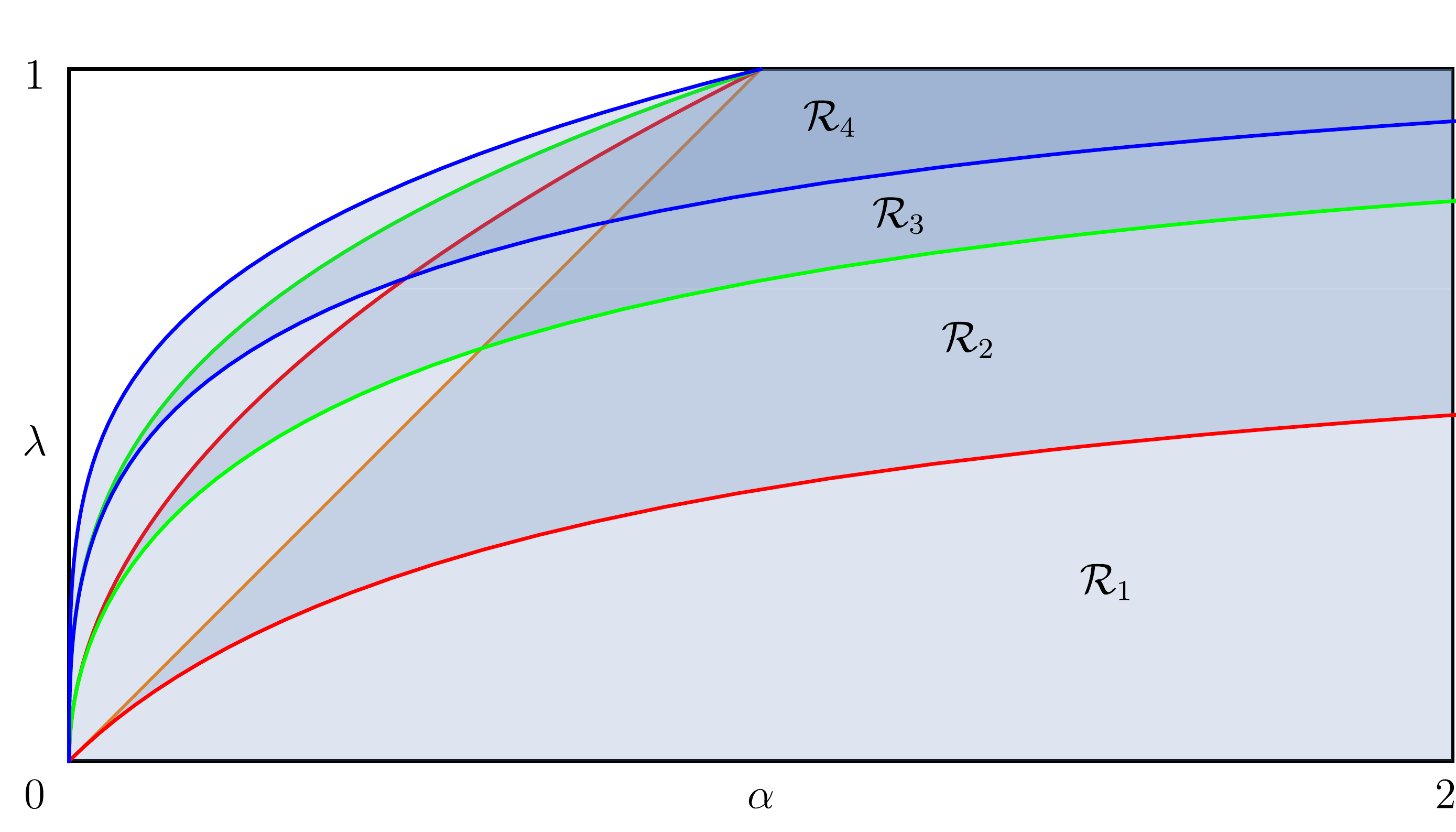}}
\caption{\footnotesize{Plot of the regions $\mathcal{R}_1,\ldots,\mathcal{R}_4$ inside the rectangle $(0,2]\times(0,1)\subset \mathcal{P}$. The region $\mathcal{R}_1$ is delimited by the brown straight line, the region $\mathcal{R}_2$ is delimited by the red curves, the region $\mathcal{R}_3$ is delimited by the green curves and the region $\mathcal{R}_4$ is delimited by the blue curves. In the parameter plane $\mathcal{P}$ only the regions $\mathcal{R}_1$ and $\mathcal{R}_2$ are unbounded.}}
\label{fig:regions_R_k}
\end{figure}

Let 
\begin{equation}\label{r}
r_\alpha(\lambda):=\frac{1+2\alpha-\alpha\lambda-\sqrt{\alpha ^2 \left(4-3 \lambda ^2\right)+\alpha  (4-6 \lambda )+1}}{2 (\alpha  \lambda -\alpha +1)}.
\end{equation}
\begin{lem}\label{lem:r}For every $(\alpha,\lambda)\in\mathcal{P}$,
\begin{enumerate}
\item $0<r_\alpha(\lambda)<\lambda$ and $r_\alpha(\lambda)<\alpha\lambda$;
\item $r_\alpha(\lambda)$ is the unique root of $q_{\alpha,\lambda}$ in the interval $(0,1)$;
\item $q_{\alpha,\lambda}(x)<0$ for $x\in(0,1)$  if and only if $x>r_\alpha(\lambda)$.
\end{enumerate}
\end{lem}
\begin{proof}
We omit the proof as it is a straightforward analysis of the function $r_\alpha$.
\end{proof}
From this lemma it is immediate that 
$$
\mathcal{R}_k=\{(\alpha,\lambda)\in\mathcal{P}\colon r_\alpha(\lambda)<\lambda^k<\alpha\},\quad k\in\Nn.
$$
Therefore, the fixed points of $P$ appear consecutively, in the sense of the following lemma.

\begin{lem}\label{lem:consec}
If $w_k$ and $w_{k'}$ are fixed points of $P$ with $k\leq k'$, then $w_j$ is a fixed point of $P$ for every $k\leq j\leq k'$.
\end{lem}

\begin{proof}
It follows from Lemma~\ref{lem: w_k fixed} that, given $(\alpha,\lambda)\in\mathcal{R}_k\cap \mathcal{R}_{k'}$ with $k\leq k'$, we have to show that $(\alpha,\lambda)\in\mathcal{R}_{j}$ for every $k\leq j\leq k'$. Indeed, since $k\leq j\leq k'$ we have $\lambda^j\leq \lambda^k<\alpha$ and $r_\alpha(\lambda)<\lambda^{k'}\leq \lambda^j$. This shows that $(\alpha,\lambda)\in \mathcal{R}_j$.
\end{proof}


Since the fixed points of $P$ form a chain in $B$, we denote by $h_\alpha(\lambda)$ the \textit{head} of the chain, i.e., the first $k\in\Nn$ for which $w_k$ is a fixed point of $P$,
$$
h_\alpha(\lambda):=\min\{k\in\Nn\colon (\alpha,\lambda)\in \mathcal{R}_k\}.
$$
Similarly, we define $t_\alpha(\lambda)$ to be the \textit{tail} of the chain, i.e., 
$$
t_\alpha(\lambda):=\max\{k\in\Nn\colon (\alpha,\lambda)\in \mathcal{R}_k\}.
$$
By Lemma~\ref{lem:consec}, the number of distinct fixed points of $P$ is
$$N_\alpha(\lambda):=t_\alpha(\lambda)-h_\alpha(\lambda)+1.$$
The following lemma gives explicit formulas for the head and tail functions and describes their asymptotic behaviour.
\begin{lem}\label{lem:ht}
For every $\alpha>0$, the head and tail functions are monotonically increasing in $\lambda$. Moreover,
\begin{enumerate}
\item 
$$
h_\alpha(\lambda)=m_\alpha(\lambda)=\begin{cases}1,&\alpha\geq1\\
\left\lfloor \log_\lambda{\alpha}\right\rfloor+1,&\alpha<1
\end{cases};
$$
\item 
$$
t_\alpha(\lambda)=\left\lceil\log_\lambda r_\alpha(\lambda)\right\rceil-1;
$$
\item 
$$
\lim_{\lambda\to 0}N_\alpha(\lambda)=1;
$$
\item if $\alpha\geq 1$, then
$$ \lim_{\lambda\to 1} N_\alpha(\lambda)=\begin{cases}\left\lceil\frac{3\alpha}{\alpha-1}\right\rceil-1,& \alpha>1\\\infty,&\alpha=1\end{cases};$$
\item if $\alpha<1$, then
\begin{align*}
\liminf_{\lambda\to 1}N_\alpha(\lambda) &= \left\lceil\frac{1+\alpha+\alpha^2}{1-\alpha}\right\rceil-1,\\
 \limsup_{\lambda\to 1}N_\alpha(\lambda)&= \left\lceil\frac{1+\alpha+\alpha^2}{1-\alpha}\right\rceil.
\end{align*}
\end{enumerate}
\end{lem}

\begin{proof}
The monotonicity of $h_\alpha$ and $t_\alpha$ follows from items (1) and (2). Now we prove the claim in each item:

\begin{enumerate}
\item Given $(\alpha,\lambda)\in\mathcal{P}$, the head $h_\alpha(\lambda)$ is the smallest natural number $k$ such that $k\in I_{\alpha,\lambda}:=(\log_\lambda\alpha,\log_\lambda{r_\alpha(\lambda)})$. By the 2nd inequality of Lemma~\ref{lem:r}, the open interval $I_{\alpha,\lambda}$ has length greater than 1. Hence, 
$$
h_\alpha(\lambda)=\max\{1,\left\lfloor \log_\lambda{\alpha}\right\rfloor+1\}.
$$
\item Similarly, the tail $t_\alpha(\lambda)$ is the greatest natural number $k$ such that $k\in I_{\alpha,\lambda}$. Because $r_\alpha(\lambda)<\lambda$ by Lemma~\ref{lem:r}, we have $\log_\lambda r_\alpha(\lambda)>1$. Thus $t_\alpha(\lambda)=\left\lceil\log_\lambda r_\alpha(\lambda)\right\rceil-1$.
\item Immediate from (1) and (2). 
\item When $\alpha\geq1$, $$N_\alpha(\lambda)=\left\lceil\log_\lambda r_\alpha(\lambda)\right\rceil-1.$$ Notice that $r_\alpha(1)=1$. Taking the limit using L'H\^opital's rule we get
$$
\lim_{\lambda\to 1}\frac{\log{r_\alpha(\lambda)}}{\log\lambda}=r_\alpha'(1)=\frac{3\alpha}{\alpha-1}.
$$
\item When $\alpha<1$, 
$$
N_\alpha(\lambda)=\left\lceil\log_\lambda r_\alpha(\lambda)\right\rceil-\left\lfloor \log_\lambda{\alpha}\right\rfloor-1.
$$
By the properties\footnote{For every $x\in\Rr$, $x\leq \lfloor x\rfloor<x+1$ and $x-1<\lceil x \rceil\leq x$.} of the floor and ceiling functions, we have the lower and upper bound
$$
\log_\lambda (r_\alpha(\lambda)/\alpha)-1\leq N_\alpha(\lambda)<\log_\lambda (r_\alpha(\lambda)/\alpha)+1.
$$
Notice that $r_\alpha(1)=\alpha$. Again, by  L'H\^opital's rule we get
$$
\lim_{\lambda\to1}\frac{\log (r_\alpha(\lambda)/\alpha)}{\log\lambda}=\frac{r_\alpha'(1)}{\alpha}=\frac{1+\alpha+\alpha^2}{1-\alpha},
$$
from which the conclusion follows.
\end{enumerate}
\end{proof}

Let $\Upsilon_k$ denote the periodic orbit of $T$ passing through $w_k$. By Lemma~\ref{lem: w_k fixed}, such orbit exists whenever $(\alpha,\lambda)\in \mathcal{R}_k$. Moreover, according to Lemma~\ref{le:periodic}, the periodic orbit $\Upsilon_k$ has period $3k$, since $\sigma^k(w_k)=1$. 

We summarize the discussion of this section in the following theorem that, together with Lemma~\ref{lem:ht}, completely characterizes the bifurcations of the attractor of $T$ and its limit as $\lambda\to1$. 

Denote by $\Lambda_\alpha(\lambda)$  the attractor of $T$. 
\begin{thm}\label{thm:bif}
For every $(\alpha,\lambda)\in\mathcal{P}$, the attractor $\Lambda_\alpha(\lambda)$ is the union of $N_\alpha(\lambda)$ distinct periodic orbits $\Upsilon_{k}$ for $k= h_\alpha(\lambda),\ldots, t_\alpha(\lambda)$. Each periodic orbit $\Upsilon_k$ has period $3k$. 

Moreover, in the symmetric case ($\alpha=1$) the number of periodic orbits $N_{\alpha}(\lambda)$ grows to infinity as $\lambda\to 1$ whereas, in the non-symmetric case ($\alpha\neq 1$), the number of periodic orbits $N_\alpha(\lambda)$ stays bounded as $\lambda\to1$, i.e.,
$$
\lim_{\lambda\to1}N_\alpha(\lambda)= \left\lceil\frac{3\alpha}{\alpha-1}\right\rceil-1,\quad \alpha>1
$$
and
$$
\liminf_{\lambda\to 1}N_\alpha(\lambda) = \left\lceil\frac{1+\alpha+\alpha^2}{1-\alpha}\right\rceil-1,\quad
 \limsup_{\lambda\to 1}N_\alpha(\lambda)= \left\lceil\frac{1+\alpha+\alpha^2}{1-\alpha}\right\rceil,
 $$
when $\alpha<1$.
Regarding the limit of the attractor as $\lambda\to1$ we have,
$$
\lim_{\lambda\to1} \Lambda_\alpha(\lambda)=\begin{cases}E,&\alpha\geq1\\S_\alpha,&\alpha<1\end{cases},
$$
where $E=\left(\frac13,\frac13,\frac13\right)$ is the Nash equilibrium of the game and $S_\alpha$ the Shapley triangle, i.e., the triangle with vertices $\left\{v_\alpha, S(v_\alpha),S^2(v_\alpha)\right\}$ where $v_\alpha=\frac{1}{1+\alpha+\alpha^2}(\alpha,1,\alpha^2)$. The limit in the previous expression is interpreted in the Hausdorff metric.
\end{thm}

\begin{proof}
It remains to prove the claim regarding the limit of the attractor.  In the case $\alpha\geq1$ we have $h_\alpha(\lambda)=1$ and $t_\alpha(\lambda)\leq  \frac{3\alpha}{\alpha-1}$. Using the expression in Lemma~\ref{lem: w_k fixed} we get,
$$
\lim_{\lambda\to1} w_{\frac{3\alpha}{\alpha-1}}=\left(\frac13,\frac13,\frac13\right).
$$
When $\alpha<1$, we have by Lemma~\ref{lem:ht},
$$
\log_\lambda\alpha\leq h_\alpha(\lambda)\leq t_\alpha(\lambda)\leq \log_\lambda r_\alpha(\lambda).
$$
Using again the expression in Lemma~\ref{lem: w_k fixed} we conclude that,
$$
\lim_{\lambda\to1} w_{\log_\lambda\alpha}=\lim_{\lambda\to1}w_{\log_\lambda r_\alpha(\lambda)}=\frac{1}{1+\alpha+\alpha^2}(\alpha,1,\alpha^2).
$$
This completes the proof of the theorem.
\end{proof}

\subsection{Phase portraits and basins of attraction}
Notice that the number $N(\varepsilon)$ of periodic strategies referred in the introduction equals $N_\alpha(\lambda)$,
where $\lambda =1-\varepsilon$ as defined in Section~\ref{sec:prelim}.

By Lemma~\ref{lem:ht},
the number of periodic orbits of the map $T$ in the symmetric case ($\alpha=1$) is given by
$$ N_1(\lambda)= \left\lceil\log_\lambda \left( \frac{3-\lambda-\sqrt{3 (1-\lambda ) (3+\lambda )}}{2\lambda} \right) \right\rceil-1, $$
Notice that $h_1(\lambda)=1$ and $N_1(\lambda)\nearrow\infty$ as $\lambda\to1$.
In Figure~\ref{fig:card_per_orb_sym} we plot the graph of $N_1(\lambda)$, where we can see for each $\lambda\in (0,1)$ the corresponding number of periodic orbits.

\begin{figure}[h]
\centering{\includegraphics[width=8cm]{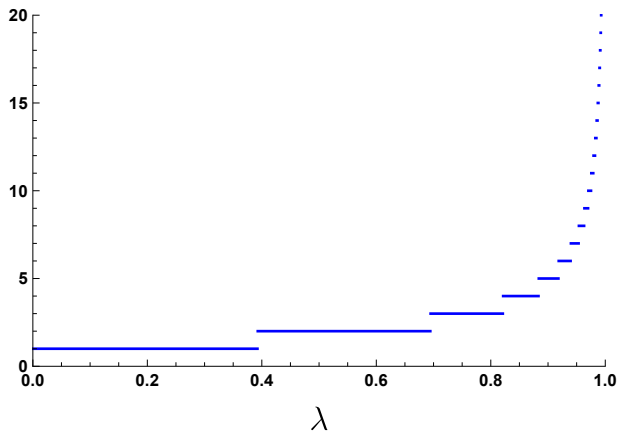}}
\caption{\footnotesize{Graph of $N_1(\lambda)$. The sequence $\{\lambda_n\}_{n\geq1}$ of discontinuities of $N_1(\lambda)$ is defined by the equation $r_1(\lambda_n )=\lambda_n^{n+2}$, $n\in\Nn$, which can be solved to give an approximate value of $\lambda_n$, e.g., $\lambda_1 \approx 0.39265$ and $\lambda_2 \approx 0.69461$. }}
\label{fig:card_per_orb_sym}
\end{figure}

For instance, when $\lambda =\frac45$ we have that $N_1(\frac45)=3$ (see Figure~\ref{fig:card_per_orb_sym} and Figure~\ref{fig:sym_08}).
In Figure~\ref{fig:basins_sym} we present the basins of attraction of the corresponding periodic orbits for the symmetric case ($\alpha =1$) for some values of the parameter $\lambda$, namely for
$\lambda =\frac45$ (where $N_1(\frac45)=3$), for $\lambda =\frac{5}{6}$ (where $N_1(\frac{5}{6})=4$),
for $\lambda =\frac{25}{28}$ (where $N_1(\frac{25}{28})=5$), and
for $\lambda =\frac{25}{27}$ (where $N_1(\frac{25}{27})=6$).

\begin{figure}
    \centering
    \begin{subfigure}[t]{0.4\textwidth}\centering
        \includegraphics[width=5cm]{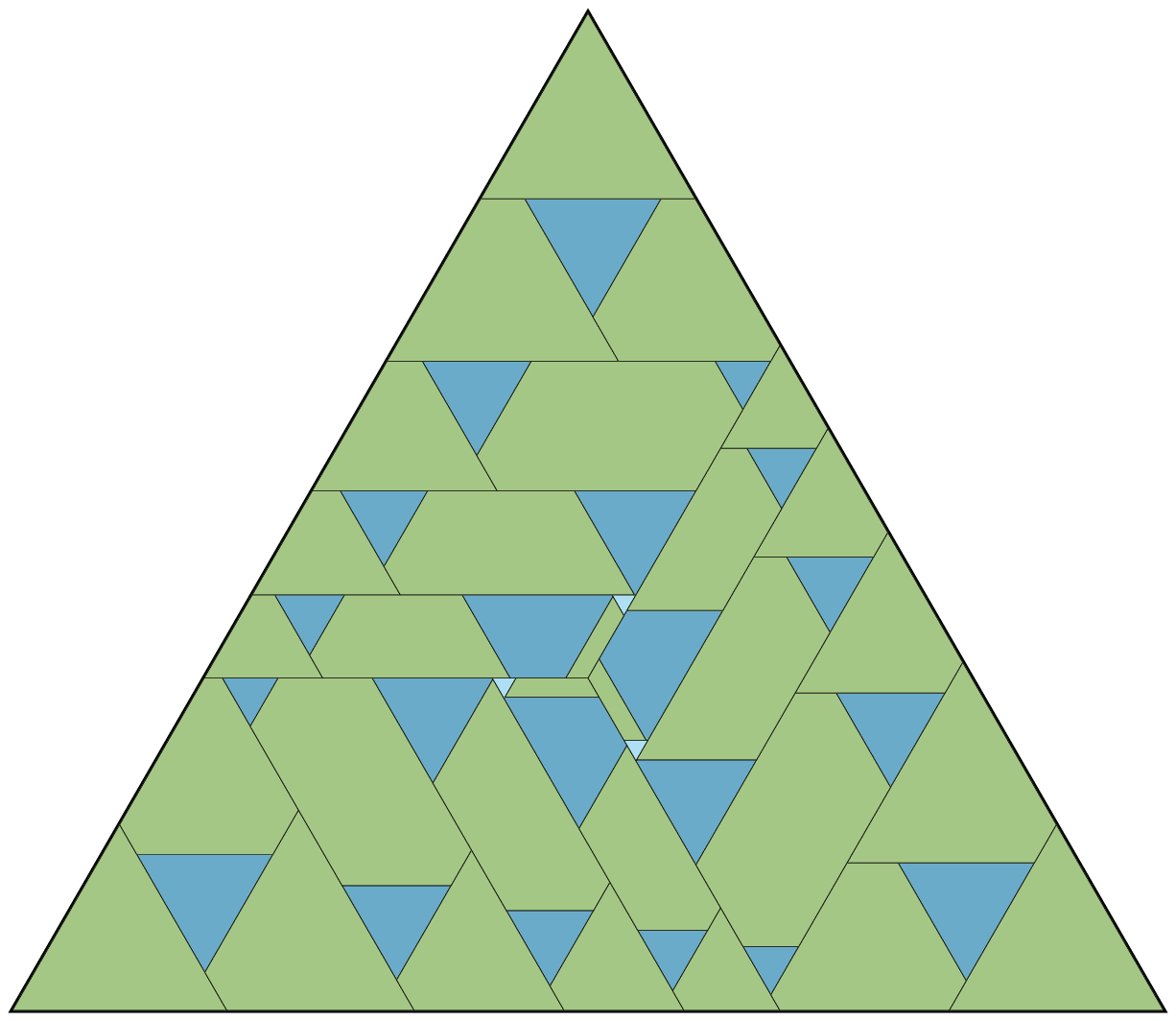}
        \caption[Short caption]{\centering $\lambda =\frac45=0.8$
        
        ($3$ periodic orbits) }
        \label{fig:sym_08}
    \end{subfigure}
    \quad \quad
    ~ 
    \begin{subfigure}[t]{0.4\textwidth}\centering
        \includegraphics[width=5cm]{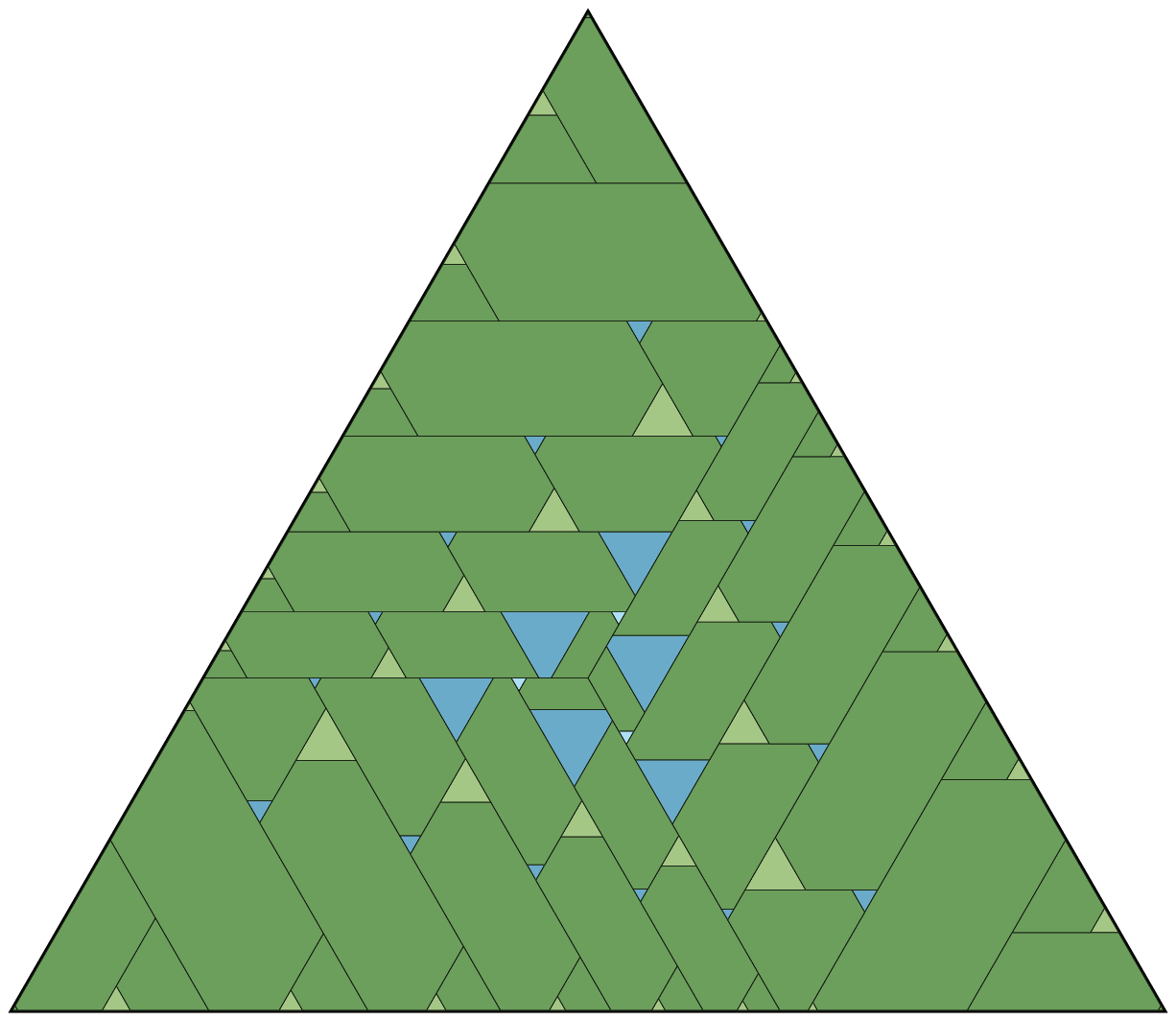}
        \caption[Short caption]{\centering $\lambda =\frac{5}{6}\approx 0.832$
        
        ($4$ periodic orbits)}
        \label{fig:sym_0832}
    \end{subfigure}
   
    \begin{subfigure}[t]{0.4\textwidth}\centering
        \includegraphics[width=5cm]{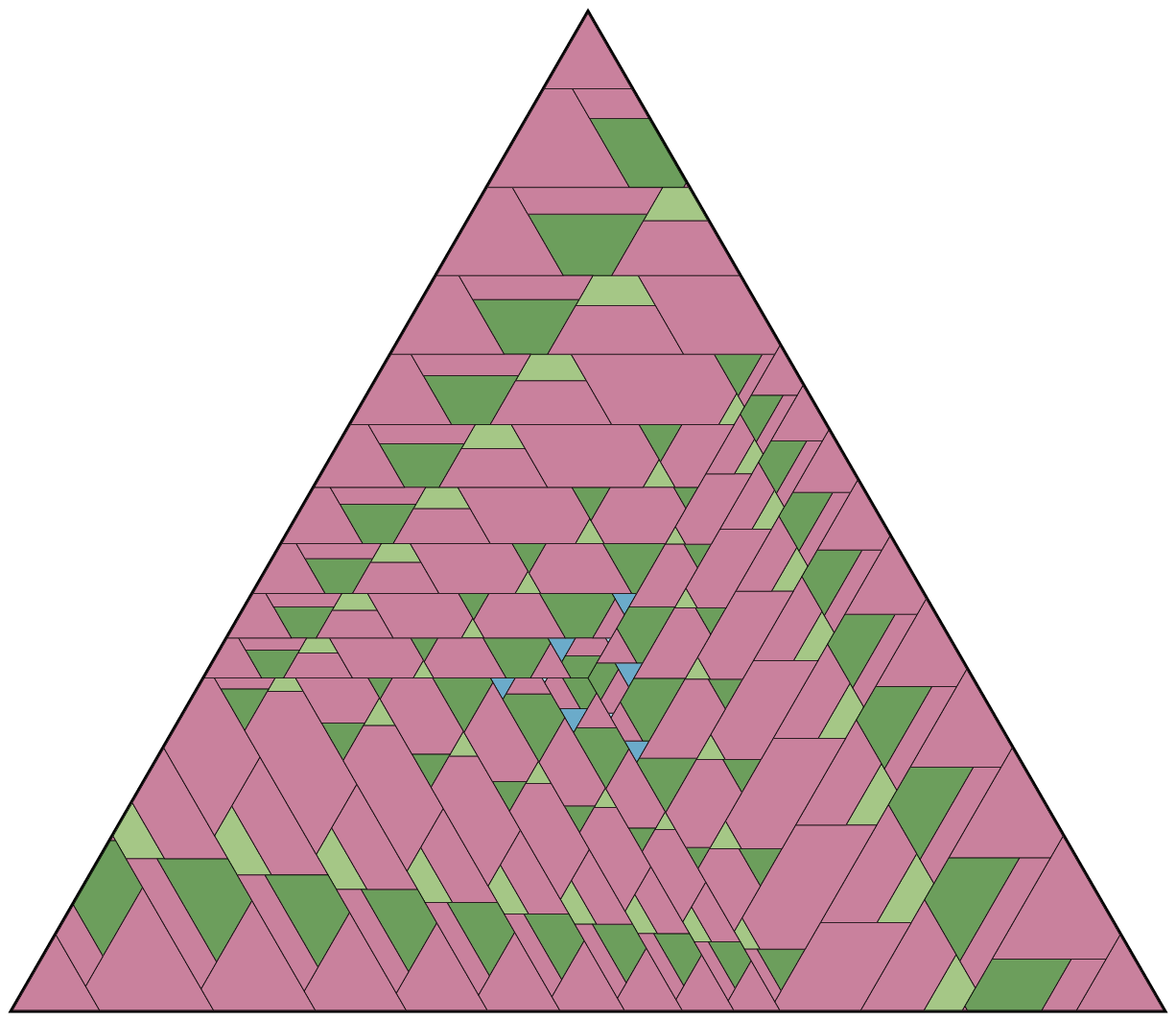}
        \caption[Short caption]{\centering $\lambda =\frac{25}{28}\approx 0.893$
        
        ($5$ periodic orbits)}
        \label{fig:sym_0893}
    \end{subfigure}
    \quad \quad
    ~
 \begin{subfigure}[t]{0.4\textwidth}\centering
        \includegraphics[width=5cm]{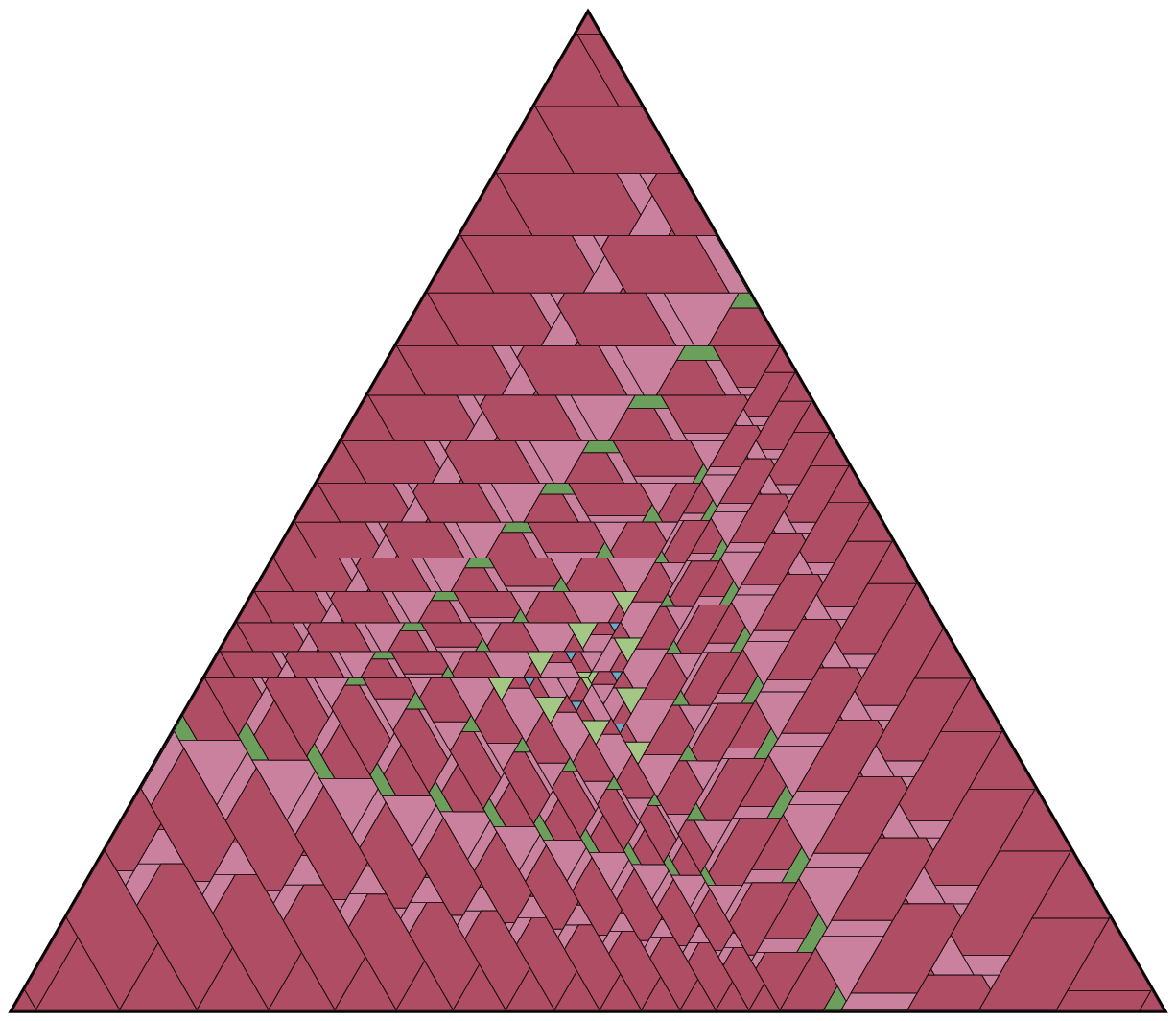}
        \caption[Short caption]{\centering $\lambda =\frac{25}{27}\approx 0.926$
        
        ($6$ periodic orbits)}
        \label{fig:sym_0926}
    \end{subfigure}
    \caption{\footnotesize{Basins of attraction of the corresponding periodic orbits for some values of the parameter $\lambda$ in the symmetric case ($\alpha =1$), where the color \mychar{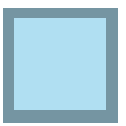} is the basin of attraction of the periodic orbit with period $3$, color \mychar{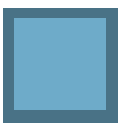} for period $6$,
\mychar{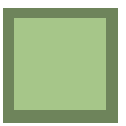} for $9$, \mychar{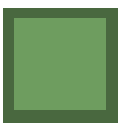} for $12$,
\mychar{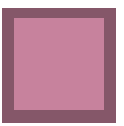} for $15$, and \mychar{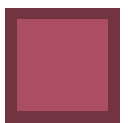} for $18$.}}\label{fig:basins_sym}
\end{figure}

In the non-symmetric favourable case $(\alpha >1)$, by $(4)$ of Lemma~\ref{lem:ht}, the plot of $N_\alpha$ is similar to $N_1$ but with a finite number $\left\lceil\frac{3\alpha}{\alpha-1}\right\rceil-1$ of plateaus (whence discontinuities). 
In the non-symmetric unfavourable case $(\alpha <1)$, by $(5)$ of Lemma~\ref{lem:ht}, the number of periodic orbits $N_\alpha(\lambda)$ oscillates in the limit $\lambda\to1$ between the values
$$ \left\lceil\frac{1+\alpha+\alpha^2}{1-\alpha}\right\rceil-1 \quad\textrm{and}\quad \left\lceil\frac{1+\alpha+\alpha^2}{1-\alpha}\right\rceil.$$
The behaviour of $N_\alpha(\lambda)$ is depicted in Figure~\ref{fig:card_per_orb_non_sym}. For instance, when $\alpha =\frac{1}{2}$, the number of periodic orbits oscillates between $3$ and $4$ as $\lambda\to 1$. In Figure~\ref{fig:basins_non_sym} we present the basins of attraction of the  periodic orbits for the unfavourable game with $\alpha =\frac{1}{2}$ for the values of the parameter $\lambda\in\left\{\frac{100}{113},\frac{25}{28},\frac{100}{111},\frac{10}{11}\right\}$. In all those examples the number of periodic orbits oscillates between $2$ and $3$ since the first $\lambda$ for which $N_{\frac12}(\lambda)=4$ is approximately $0.953613$. However, computing and colouring the basins of attraction of the periodic orbits (the least period being 45) demands for a greater computational effort compared to the examples computed in Figure~\ref{fig:basins_non_sym}. Indeed, the number of elements in the partition forming the basins of attraction of the periodic orbits is 2133.

\begin{figure}[h]
\centering{\includegraphics[width=9cm]{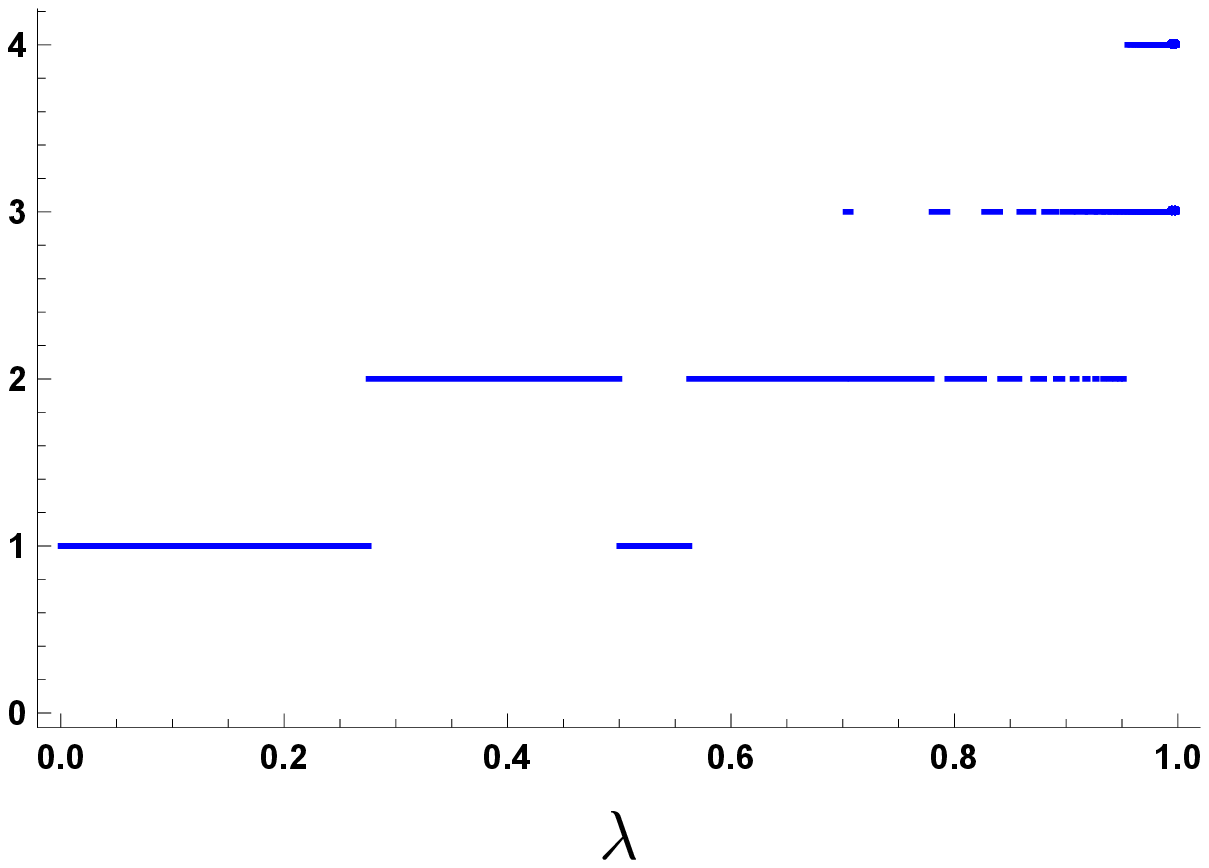}}
\caption{\footnotesize{Graph of $N_{\frac{1}{2}}(\lambda)$.}}
\label{fig:card_per_orb_non_sym}
\end{figure}



\begin{figure}
    \centering
    \begin{subfigure}[t]{0.4\textwidth}\centering
        \includegraphics[width=5cm]{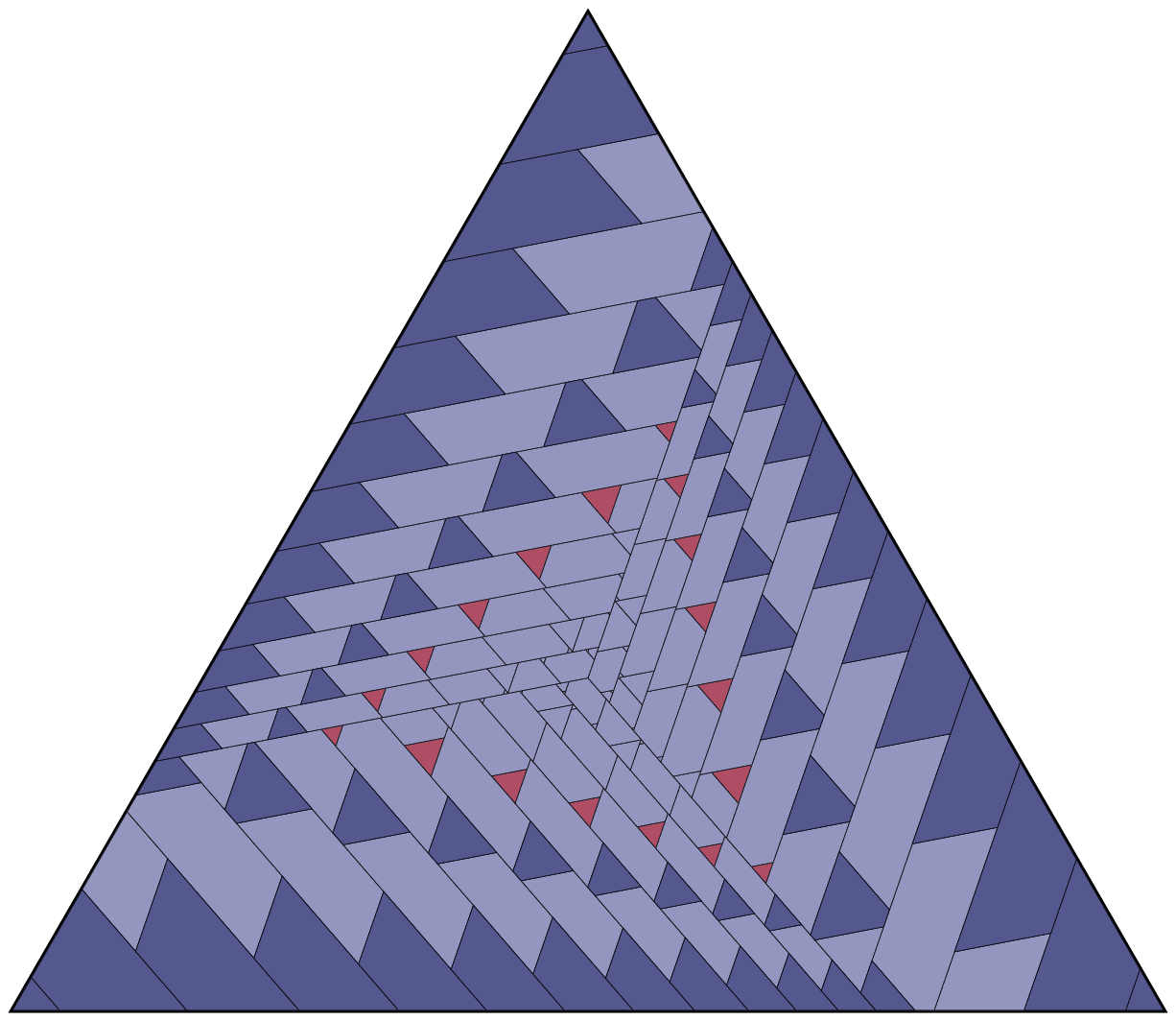}
        \caption[Short caption]{\centering $\lambda =\frac{100}{113}\approx 0.885$
        
        ($3$ periodic orbits) }
        \label{fig:non_sym_0885}
    \end{subfigure}
    \quad \quad
    ~ 
    \begin{subfigure}[t]{0.4\textwidth}\centering
        \includegraphics[width=5cm]{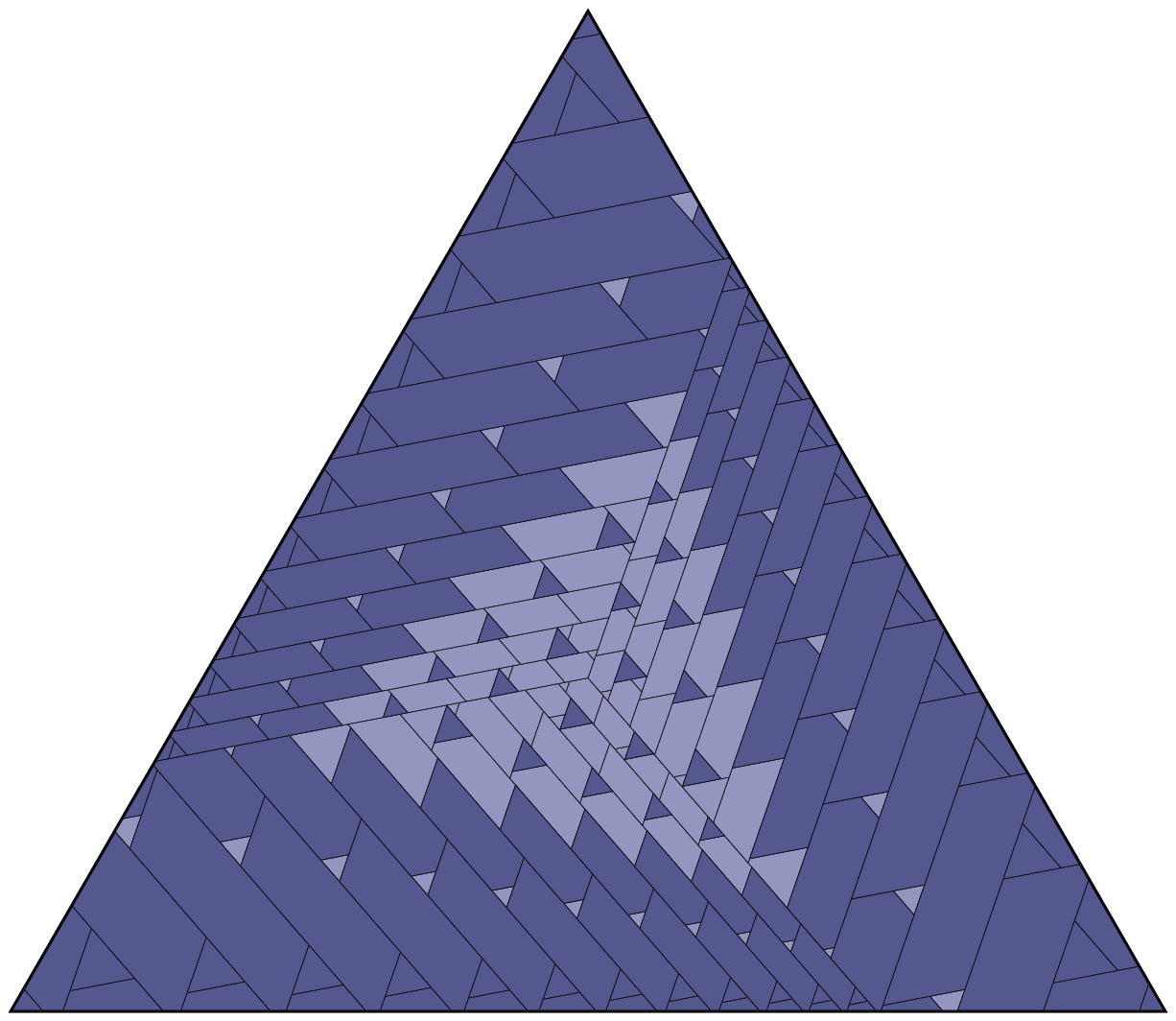}
        \caption[Short caption]{\centering $\lambda =\frac{25}{28}\approx 0.893$
        
        ($2$ periodic orbits)}
        \label{fig:non_sym_0893}
    \end{subfigure}
   
    \begin{subfigure}[t]{0.4\textwidth}\centering
        \includegraphics[width=5cm]{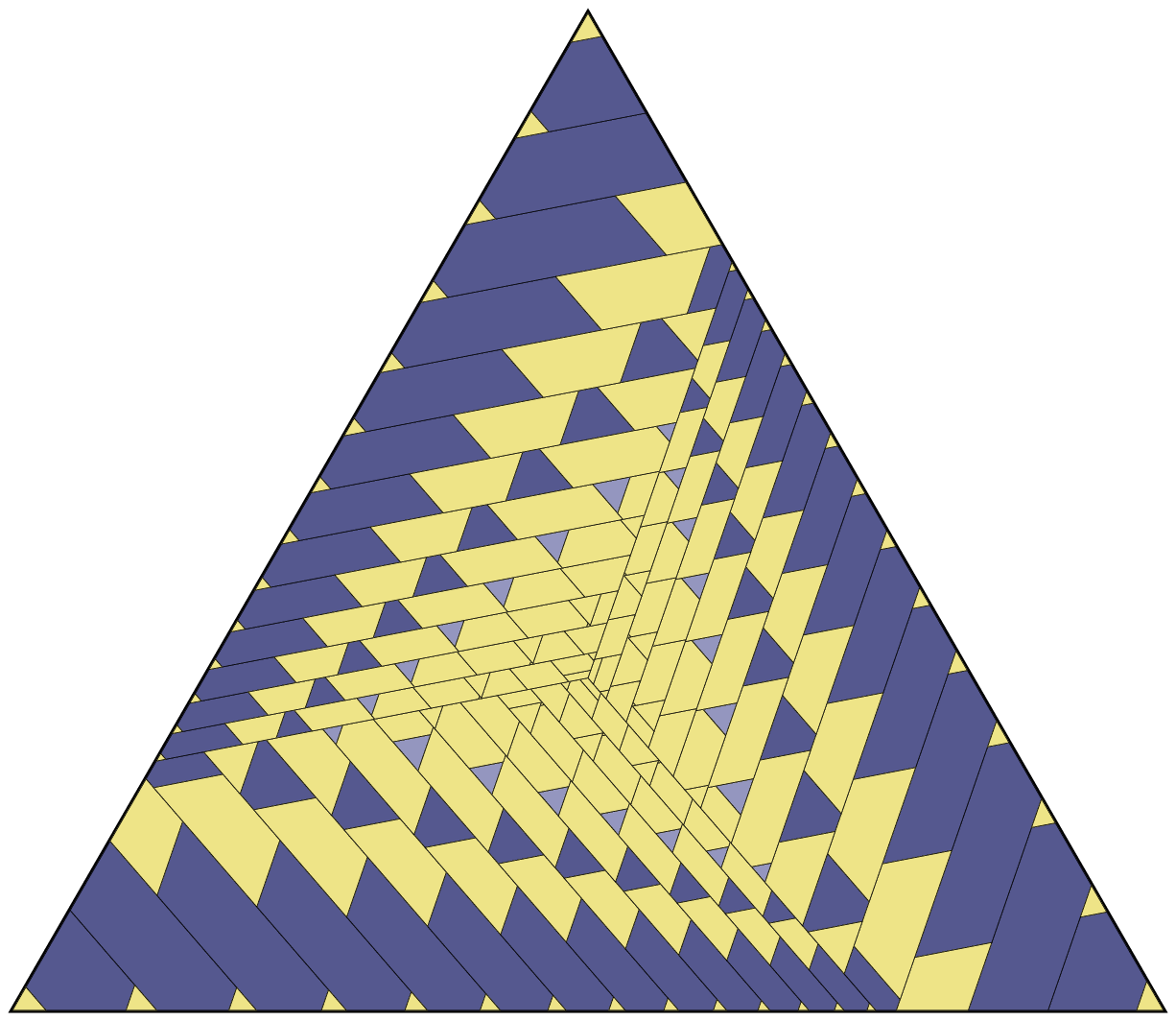}
        \caption[Short caption]{\centering $\lambda =\frac{100}{111}\approx 0.901$
        
        ($3$ periodic orbits)}
        \label{fig:non_sym_0901}
    \end{subfigure}
    \quad \quad
    ~
 \begin{subfigure}[t]{0.4\textwidth}\centering
        \includegraphics[width=5cm]{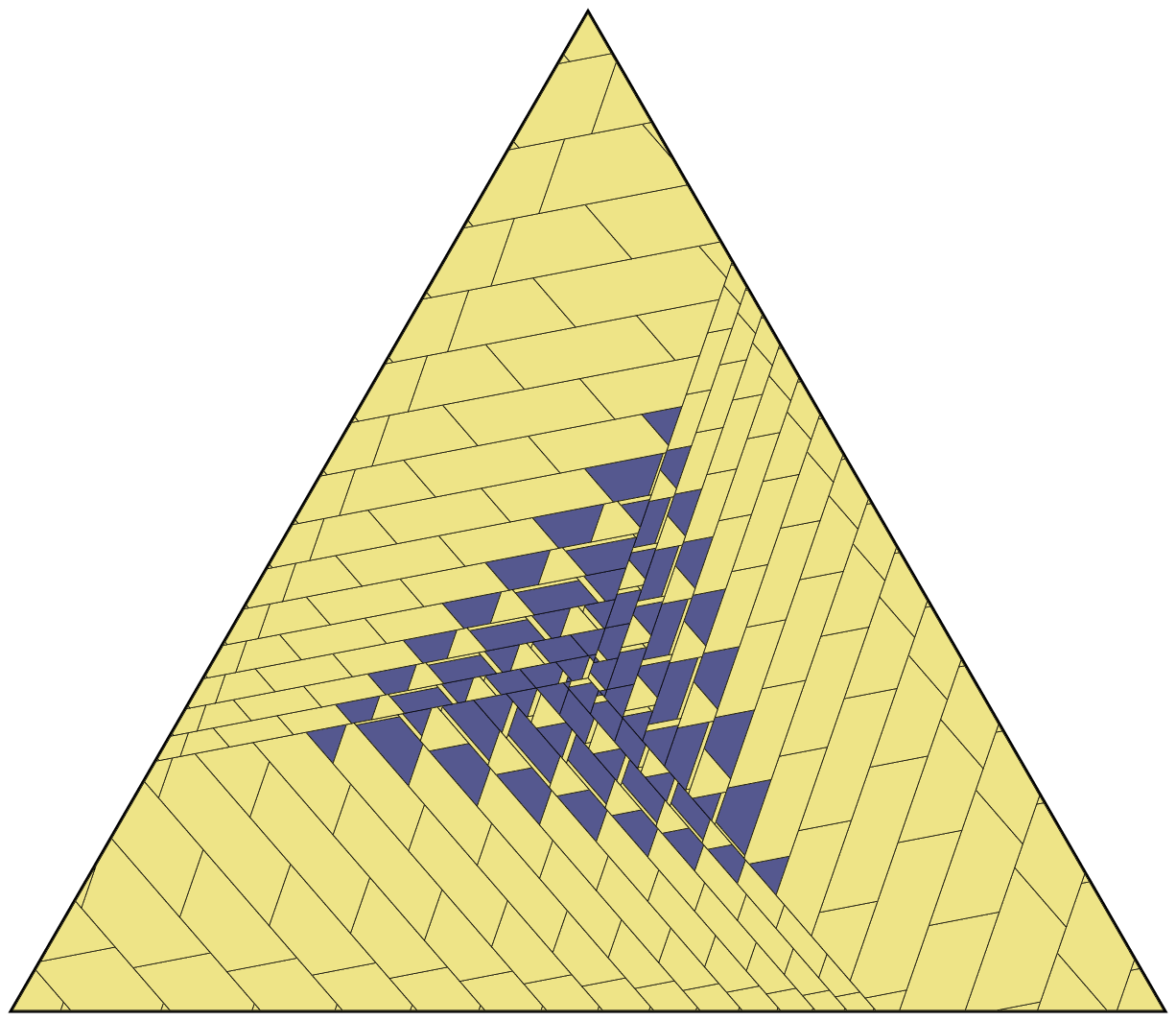}
        \caption[Short caption]{\centering $\lambda =\frac{10}{11}\approx 0.909$
        
        ($2$ periodic orbits)}
        \label{fig:non_sym_0909}
    \end{subfigure}
    \caption{\footnotesize{Basins of attraction of the corresponding periodic orbits for some values of the parameter
$\lambda$ in the non symmetric case with $a=\frac{1}{2}$ and $b=1$ ($\alpha =\frac{1}{2}$),
where the color \mychar{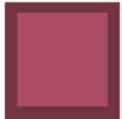} is the basin of attraction of the periodic orbit with period $18$, color \mychar{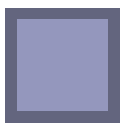} for period $21$, \mychar{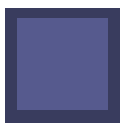} for $24$, and \mychar{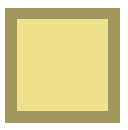} for $27$.}}\label{fig:basins_non_sym}
\end{figure}

\section{Conclusion and open problems}\label{sec:Conclus}

In this paper we show that the dynamics of the discretized best-response function for the RPS game with payoff matrix~\eqref{RPS_payoff_matrix} is asymptotically periodic.
In fact we prove that its attractor is finite and periodic in the sense that any strategy converges to a periodic strategy and there are at most a finite number of them.
Moreover, we fully characterize this dynamics as a function of the parameters $\alpha=a/b$ and $\lambda$ (related to the step of the discretization). Namely, we determine the exact number of periodic strategies, their period, and location, as a function of the parameters $\alpha$ and $\lambda$. We believe that our methods can be applied to study the discretized best-response dynamics of other games, such as $2\times 2$ bimatrix games.

The discretized best-response dynamics of the RPS game belongs to a special class of dynamical systems with discontinuities called planar piecewise affine contractions. In \cite{BD09} it is proved that a generic planar piecewise affine contraction is asymptotically periodic, i.e., the attractor is finite and periodic for almost every choice of the branch fixed points of the affine contractions. In our case, the branch fixed points are pure strategies and thus cannot be used as a varying parameter. The dynamics of other classes of piecewise contractions has been recently investigated in \cite{G18,NPR18,CGMU16,DGG15,NP15,NPR14}. 
 
A paradigmatic one-dimensional piecewise contraction that illustrates the typical dynamical behaviour that one can observe is the contracted rotation $f:[0,1)\to[0,1)$ defined by
$$
f(x)=\lambda x+b\pmod{1}
$$
with parameters $\lambda\in(0,1)$ and $b\in(0,1)$ such that $\lambda+b>1$ \cite{LN2018}. The map $f$ can be seen as a circle map with a single discontinuity point at $x=0$. Then, associated to $f$, one can define the rotation number $\rho(f)$ which describes the asymptotic rate of rotation of $f$. Depending on the arithmetical properties of $\rho(f)$, the map $f$ can display distinct dynamical phenomenon. When $\rho(f)$ is rational, $f$ has a unique periodic orbit which is a global attractor. On the other hand, when $\rho(f)$ is irrational, $f$ is quasi-periodic, in the sense that the global attractor is a Cantor set. 
It would be interesting to construct a game for which the discretized best-response dynamics is quasi-periodic.

Based on these works we conjecture that for any discretization step and for almost every payoff matrix $A$, the corresponding discretized best-response dynamics is asymptotically periodic.

Another interesting line of research would be to consider variable contraction rates
$$
x_{n+1}=\lambda_{n} x_n+(1-\lambda_{n})BR(x_n)
$$
where $\lambda_n\in(0,1)$ for every $n\geq0$ which are either deterministic or random. In the deterministic case, the contraction rates $\lambda_n$ could be generated by a dynamical system $\lambda_{n+1}=g(\lambda_n)$ where $g:[0,1]\to[0,1]$ is an interval map which models how the fraction of the population that chooses to play the same strategy varies over time. In the case of the original Brown's fictitious play, $\lambda_n=1-\frac{1}{n}$ which is the orbit of $g(\lambda)=\frac{1}{2-\lambda}$ starting at $\lambda_1=0$. Notice that the point $\lambda=1$ is a neutral fixed point of $g$. Alternatively, $\{\lambda_n\}_{n\geq0}$ could be an ergodic Markov chain and one could study the existence of ergodic stationary measures for the random discretized best-response dynamics.

\section*{Acknowledgments}
The authors were supported by FCT - Funda\c{c}\~{a}o para a Ci\^{e}ncia e a Tecnologia,
under the project CEMAPRE - UID/MULTI/00491/2019 through national funds. The authors also wish to express their gratitude to Jo\~ao Lopes Dias for stimulating conversations.


\bibliographystyle{amsplain}


\appendix
\section{Monotonicity lemmas}\label{sec:MonotLems}
Recall that
$$
m=m_\alpha(\lambda):=\min\{k\in\Nn\colon \lambda^k<\alpha\}.
$$
and that $\hat{B}$ is the set of regular strategies in $B$, i.e., $\hat{B}=B\cap \hat{\Delta}$.  

\begin{lem}\label{lem:mon1}
Let $x\in \hat{B}$. If $n(x)\geq m$, then $n(P(x))\geq m$.
\end{lem}
\begin{proof}
We suppose that $m\geq2$. Otherwise there is nothing to prove. Let $x\in B_k$ for some $k\geq m$. Then
$$
S^2(u_\alpha)\cdot P(x)=\lambda^{k} S(u_\alpha)\cdot x+(1-\lambda^{k-1})(2+\alpha).
$$
Since $x\in B_k$, we have
\begin{align*}
S(u_\alpha)\cdot x &= 3-u_\alpha\cdot x-S^2(u_\alpha)\cdot x\\
&> 3-(\lambda^{-1}\alpha-\alpha+1)-b_k.
\end{align*}
Hence,
$$
S^2(u_\alpha)\cdot P(x)> 2+\alpha -\lambda^{k-1}-\alpha\lambda^{-1}.
$$
Therefore, to prove that $n(P(x))\geq m$, it is sufficient to show that 
$$
g(\lambda):=2+\alpha -\lambda^{k-1}-\alpha\lambda^{-1}-b_{m-1}>0.
$$
Notice that
$$
g(\lambda)=\lambda^{-1}(\alpha+1-\lambda^{k}-\lambda^{-m+1}\alpha).
$$
Since $\lambda^m< \alpha$ and $\lambda^k\leq\lambda^m$ we get,
$$
g(\lambda)>\lambda^{-1}(1-\lambda^{-m+1}\alpha).
$$
Again, by the definition of $m$, we have $\lambda^{m-1}\geq\alpha$. This shows that $g(\lambda)>0$ as we wanted to prove.
\end{proof}
\begin{lem}\label{lem:mon2}
Let $x\in \hat{B}$ such that $n(P(x))\geq m$. If $n(P(x))\le n(x)$, then $n(P^2(x))\le n(P(x))$.
\end{lem}
\begin{proof}
We want to prove that given $k\in\Nn$, if $x\in B_k$ and $P(x)\in B_j$ for some $j\le k$, then
$P^2(x)\in B_i$ for some $i\le j$,
for any $\alpha>0$ and $\lambda\in(0,1)$ satisfying $\lambda^j< \alpha$.

So its enough to see that
$$
S^2(u_\alpha)\cdot P^2(x)<b_j=\lambda^{-j-1}\alpha-\lambda^{-1}(\alpha-1+(1-\lambda)(\alpha+2)).
$$
Since $x\in B_k$,
$$
P(x)=\lambda^k S(x)+\lambda^{k-1}(1-\lambda)e_1+(1-\lambda^{k-1})e_2
$$
and $P(x)\in B_j$ implies that
\begin{align*}
P^2(x) &=\lambda^j S(P(x))+\lambda^{j-1}(1-\lambda)e_1+(1-\lambda^{j-1})e_2 \\
			&=\lambda^{j+k} S^2(x)+\lambda^{j+k-1}(1-\lambda)e_3+\lambda^j(1-\lambda^{k-1})e_1+\lambda^{j-1}(1-\lambda)e_1\\
			&\qquad  +(1-\lambda^{j-1})e_2.
\end{align*}
Because $S^2(u_\alpha)=(0,2+\alpha,1-\alpha)$, we have that
$$
S^2(u_\alpha)\cdot P^2(x) = \lambda^{j+k} u_\alpha\cdot x-\lambda^{j+k-1}(1-\lambda)(\alpha-1)+(1-\lambda^{j-1})(2+\alpha).
$$
But $u_\alpha\cdot x<\alpha\left(\lambda^{-1}-1\right)+1$, thus
$$
S^2(u_\alpha)\cdot P^2(x)<Q_{j,k,\alpha}(\lambda),
$$
where
$$
Q_{j,k,\alpha}(\lambda):=\lambda^{j+k} \left( \lambda^{-1}\alpha-\alpha+1 \right) -\lambda^{j+k-1}(1-\lambda)(\alpha-1)+(1-\lambda^{j-1})(2+\alpha).
$$
Now it is easy to see that $Q_{j,k,\alpha}(\lambda)<b_j$ for every $\alpha>0$ and $\lambda\in(0,1)$ satisfying $\lambda^j<\alpha $. Indeed, 
\begin{align*}
\lambda^{j+1}( Q_{j,k,\alpha}(\lambda)-b_j)&=\lambda ^j \left(\lambda ^{j+k}-2 \lambda ^j+1\right)-\alpha  \left(\lambda ^j-1\right)^2\\
			&\le\lambda ^j \left(\lambda ^{2j}-2 \lambda ^j+1\right)-\alpha  \left(\lambda ^j-1\right)^2\\
			&= (\lambda^j-1)^2(\lambda^j-\alpha)<0.
\end{align*}
\end{proof}

\begin{lem}\label{lem:mon3}
Let $x\in \hat{B}$ such that $n(P(x))<m$. If $n(P(x))\ge n(x)$, then $n(P^2(x))\ge n(P(x))$.
\end{lem}

\begin{proof}
The strategy of the proof is the same as that of the previous lemma.
We want to prove that given $k\in\Nn$, if $x\in B_k$ and $P(x)\in B_j$ for some $k\leq j<m$, then
$P^2(x)\in B_i$ for some $i\ge j$,
for any $\alpha>0$ and $\lambda\in(0,1)$ satisfying $\lambda^j\geq \alpha$.

It is enough to see that
$$
S^2(u_\alpha)\cdot P^2(x)> b_{j-1}=\lambda^{-j}\alpha-\lambda^{-1}(\alpha-1+(1-\lambda)(\alpha+2)).
$$
As in the proof of the previous lemma we have that
\begin{align*}
S^2(u_\alpha)\cdot P^2(x) &= \lambda^{j+k} u_\alpha\cdot x-\lambda^{j+k-1}(1-\lambda)(\alpha-1)+(1-\lambda^{j-1})(2+\alpha).
\end{align*}
But $u_\alpha\cdot x>1$, so
$$
S^2(u_\alpha)\cdot P^2(x)>Q_{j,k,\alpha}(\lambda),
$$
where
$$
Q_{j,k,\alpha}(\lambda):=\lambda^{j+k} -\lambda^{j+k-1}(1-\lambda)(\alpha-1)+(1-\lambda^{j-1})(2+\alpha).
$$
Now it is easy to see that $Q_{j,k,\alpha}(\lambda)>b_{j-1}$ for every $\alpha>0$ and $\lambda\in(0,1)$ satisfying $\lambda^j\geq \alpha $. Indeed, 

\begin{align*}
g_{j,k,\alpha}(\lambda)&:=\lambda^{j+1}(Q_{j,k,\alpha}(\lambda)-b_{j-1}) \\
		   & = \alpha\lambda^{2j+k+1}-(\alpha-1)\lambda^{2j+k}-(2+\alpha)\lambda^{2j}+(1+2\alpha)\lambda^j-\alpha\lambda\\
		    &=\alpha(1-\lambda)(1-\lambda^{2j+k})+ \lambda^{2j+k}-(2+\alpha)\lambda^{2j}+(1+2\alpha)\lambda^j-\alpha\\
		    	&>\lambda^{2j+k}-(2+\alpha)\lambda^{2j}+(1+2\alpha)\lambda^j-\alpha\\
		    	&\ge\lambda^{3j}-(2+\alpha)\lambda^{2j}+(1+2\alpha)\lambda^j-\alpha\\
			&= \lambda^{3j}-3\lambda^{2j}+3\lambda^j-1-(\alpha-1)(\lambda^{2j}-2\lambda^j+1)\\
			&= (\lambda^j-1)^3-(\alpha-1)(\lambda^j-1)^2\\
			&= (\lambda^j-1)^2\left( \lambda^j-\alpha\right) \geq0.
\end{align*}
\end{proof}

\end{document}